\documentclass[11pt]{article}
\usepackage[a4paper,margin=1in]{geometry}
\usepackage{amsfonts,amssymb,amsmath}
\usepackage{amsthm}
\usepackage{graphicx}
\usepackage{psfrag}
\usepackage{enumerate}
\usepackage{float}
\usepackage{calc}
\usepackage{epic}
\usepackage{subfigure}
\usepackage{comment}
\usepackage{hyperref}

\makeatletter
\newcommand{\customlabel}[2]{%
   \protected@write \@auxout {}{\string \newlabel {#1}{{#2}{\thepage}{#2}{#1}{}} }%
   \hypertarget{#1}{#2}
}
\makeatother

\def\R{\mathbb{R}}
\def\X{\mathcal{X}}

\def\Measures{ \mathcal{M}(\R) }
\def\MSpacesmall{ \mathcal{X}_M }
\def\MSpacebig{ \MSpacesmall' }
\def\BSpace{ \mathcal{X}_B }
\def\LSpace{ \mathcal{X}_\Lambda }
\def\ISpace{ \mathcal{X}_I }
\def\JSpace{\mathcal{X}_J }

\def\Pc{\mathcal{P}}

\def\l{\mathcal{L}}
\def\lf{\mathcal{LF}}
\def\j{\mathcal{J}}

\def\H{H}

\def\loynes{\delta}
\def\Iy{ {I} }
\def\IM{ {I_M} }
\def\IL{ {I_\Lambda} }
\def\II{ {I_I} }
\def\ILo{ {I_{\loynes}} }
\def\IJ{{I_J}}

\def\estJ{ {J_n} }
\def\estM{ {M_n} }
\def\estI{ {I_n} }
\def\estL{ {L_n} }
\def\estLo{ {\loynes_n} }
\def\mover{ {\overline{m}} }
\def\munder{ {\underline{m}} }

\def\gl{ {\theta_{g,l} } }
\def\gr{ {\theta_{g,r} } }

\def\fl{ {\theta_{f,l} } }
\def\fr{ {\theta_{f,r} } }

\def\fBl{ {\eta_{f,l,k} } }
\def\fBr{ {\eta_{f,r,k} } }

\def\fs{ f^\sharp }
\def\fsl{ {\theta_{\fs_k,l} } }
\def\fsr{ {\theta_{\fs_k,r} } }

\def\Bc{ {\overline{B}} }

\def\xt{ {x_\theta} }
\def\xtpe{ {x_{\theta+\epsilon}} }
\def\xtme{ {x_{\theta-\epsilon}} }
\def\yt{ {y_\theta} }
\def\mup{ {f_{\mu_+}} }
\def\mum{ {f_{\mu_-}} }

\def\la{ {l_\alpha} }

\def\fmimic{ {f_{\nu}} }
\def\effdomf{ {\mathcal{D}_f }}

\def\effdomsub{ \mathcal{D}_{[\alpha,\beta]}^\subset }
\def\effdom{ {\mathcal{D}_{[\alpha,\beta]}} }

\def\estLambda{\Lambda_n}

\newlength{\noteWidth}
\long\def\notes#1{\ifinner
             {\tiny #1}
             \else
             \marginpar{\parbox[t]{\noteWidth}{\raggedright\tiny #1}}
             \fi}

\newcounter{rmnum}

\def\limsup{\mathop{\rm lim\ sup}}
\def\liminf{\mathop{\rm lim\ inf}}

\def\epi{ {\operatorname{epi}} }
\def\taw{ {\tau_{\text{AW}} } }

\def\N{ {\mathbb{N} }}

\def\Y{ {\mathcal{Y} }}

\newtheorem*{definition*}{Definition}
\newtheorem{definition}{Definition}[section]
\newtheorem{theorem}{Theorem}[section]
\newtheorem{corollary}{Corollary}[section]
\newtheorem{lemma}{Lemma}[section]
\newtheorem{proposition}{Proposition}[section]

\begin{document}
\title{Estimating large deviation rate functions}
\date{$21^{\rm st}$ August 2015}
\author{
Ken R. Duffy\thanks{
Hamilton Institute, National University of Ireland Maynooth, Ireland. 
E-mail: \texttt{ken.duffy@nuim.ie}}
\and 
Brendan D. Williamson\thanks{
Department of Mathematics, Duke University, Durham, North Carolina, USA.
E-mail: \texttt{brendan.williamson@duke.edu}}
}
\maketitle

\begin{abstract}

Establishing a Large Deviation Principle (LDP) proves to be a
powerful result for a vast number of stochastic models in many
application areas of probability theory. The key object of an LDP
is the large deviations rate function, from which probabilistic
estimates of rare events can be determined. In order make these
results empirically applicable, it would be necessary to estimate
the rate function from observations. This is the question we address
in this article for the best known and most widely used LDP:
Cram\'er's theorem for random walks.

We establish that even when only a narrow LDP holds for Cram\'er's
Theorem, as occurs for heavy-tailed increments, one gets a LDP for
estimating the random walk's rate function in the space of convex
lower-semicontinuous functions equipped with the Attouch-Wets
topology via empirical estimates of the moment generating function.
This result may seem surprising as it is saying that for Cram\'er's
theorem, one can quickly form non-parametric estimates of the function
that governs the likelihood of rare events.

\end{abstract}

\section{Introduction}

Large deviation theory \cite{Varadhan66,Dembo98}, the study of the
exponential decay in probability of unlikely events, has been used
extensively in fields such statistical mechanics
\cite{Ellis06,Touchette09}, insurance mathematics \cite{Asmussen10},
queueing systems \cite{Shwartz95,Ganesh04}, importance sampling
\cite{Dupuis04} and many others. In each of these fields, it is
typically the case that a Large Deviation Principle (LDP) is shown
to hold based on an assumed underlying stochastic model of the
process of interest. To quantify the rate of decay in the probability
of events as a function of system size, the LDP rate function, the
negative of the statistical mechanical entropy, is identified in
terms of the properties of the underlying stochastic process,
enabling direct estimates on the probability of rare events.

In many experimental systems, however, an \emph{a priori}
parameterization of the underlying stochastic nature of the system
is unknown and must be garnered from data. In these situations, to
transfer large deviation results from theory to empirical practice,
it is necessary to estimate the associated rate function from data
as a random function. It is the question of whether, in a non-parametric
setting, this is possible and, if so, what is the speed of convergence
of the estimates that is the subject considered here. 

In the most commonly used LDP, Cram\'er's theorem for random walks,
which underlies many other LDP results, we establish that non-parametric
estimation of the rate function is not only possible, but that the
probability of mis-estimation is, in appropriate sense, decaying exponentially in the
observed sample size. 

To make matters more precise, let $\{X_i\}$ be a sequence of
real-valued i.i.d. random variables, and define $S_n=(X_1+\cdots+X_n)/n$
to be the sample mean. If the Moment Generating Function (MGF)
of $X_1$, $M(\theta)=E(\exp(\theta X_1))$ for $\theta\in\R$, is
finite in a neighbourhood of the origin, then by Cram\'er's theorem
$\{S_n\}$ satisfies the LDP \cite{Varadhan66,Dembo98} with a convex
rate function $I$ that has compact level sets,
\begin{align*}
I(x) = \sup_{\theta\in\R}(\theta x - \Lambda(\theta)),
	\text{ where } \Lambda(\theta)=\log M(\theta).
\end{align*} 
Speaking roughly, for large $n$ this is suggestive of $dP(S_n= x)\asymp\exp(-n
I(x)) dx$. If, on the other hand, the MGF 
of $X_1$ is not finite in a neighbourhood of the origin, as happens
with heavy-tailed increments, then $\{S_n\}$ satisfies a narrow LDP
\cite{Dembo98}, where the rate function does not need to have compact
level sets.

If we do not know the distribution of $X_1$, but observe a sequence
$X_1,\ldots,X_n$, is it possible to create non-parametric estimates
of the rate function $I$ that are well-behaved in that they converge
quickly as a function of the sample size? Given that $I$ captures
the probability of unlikely events, the perhaps surprising answer
will prove to be yes.

Drawing parallels with chemical engineers who estimate entropy
directly rather than building parametric models, Duffield et al.
\cite{Duffield95b}, based on private communication with A. Dembo,
proposed using the logarithm of the Maximum Likelihood Estimator
(MLE) for the MGF as an estimate of the cumulant generating function
of tele-traffic streams. Even though the estimator proved resistant
to rigorous determination of its analytic properties, it seemed
practically applicable and so was put to use, e.g. \cite{Lewis98}.
Independently and a little later, a similar approach was developed
in statistical mechanics as a means of estimating equilibrium free
energy differences, where it is called Jarzynski's estimator
\cite{Jarzynski97}. Significant examples of its use in an experimental
context can be found in \cite{hummer2001,Liphardt2002,
Gupta2011,Saira2012}, with a recent theoretical study provided in
\cite{Rohwer14}.

Given observations $X_1,\ldots,X_n$, the estimator considered in
\cite{Duffield95b} is the maximum likelihood estimator of the MGF
as a random convex function:
\begin{align} 
\label{eq:estM}
\estM(\theta) = 
	\frac 1n \sum_{i=1}^n e^{\theta X_i} 
	\text{ for } \theta\in\R.  
\end{align} 
From this, their proposed estimate of the Cumulant Generating
Function (CGF) given $n$ observations is
\begin{align*}
\estLambda(\theta) = \log\estM(\theta) \text{ for } \theta\in\R,
\end{align*} 
which is also a convex function. Jarzynski's estimator, which could
be considered as an empirical estimate of the Effective Bandwidth
in teletraffic engineering \cite{Kelly96}, is given by
\begin{align} 
\label{eq:estJ}
\estJ(\theta) = \frac{1}{\theta} \estLambda(\theta) \text{ for } \theta\in\R.  
\end{align} 
Following \cite{Duffy05}, we use the Legendre-Fenchel transform of
the CGF estimator to give the following random function as an
estimate of the rate function:
\begin{align}
\label{eq:estI}
	\estI(x) = \sup_{\theta\in\R}(\theta x - \estLambda(\theta))
	\text{ for } x\in\R.
\end{align}
It is the large deviation behavior of the random functions $\{\estM\}$,
$\{\estLambda\}$, $\{\estJ\}$, and $\{\estI\}$ that is of interest
to us.

Point-estimate properties for a single fixed $\theta$ have been
established for $\{\estLambda\}$. For example, assuming $X_1$ is
bounded, \cite{Gyorfi00} provides concentration inequalities
establishing speed of convergence of the estimate, \cite{Ganesh96}
provides a means for correcting implicit bias in the estimation of
effective bandwidths, while \cite{Ganesh98, Ganesh02} provide for
a Bayesian approach. Motivated by Jarzynski's estimator, the recent
study \cite{Rohwer14} considers unbounded random variables, focusing
on the argument $\theta$ where the estimates become unreliable.  As
both our estimate of $I$, and many estimates of interest, depend upon
$\estLambda(\theta)$ for all $\theta$, however, for many applications
it is necessary to consider $\estLambda$ as a random function rather
than a random, extended real-valued estimate.

As examples of functions of interest, in the random walk case, it
is known, e.g. \cite{Glynn94,Duffy03,lelarge08,Asmussen10}, that
if $E(X_1)<0$, then the tail of the supremum of the random walk
satisfies
\begin{align}
\label{eq:loynes}
\lim_{q\to\infty}\frac1q\log P\left(\sup_{k\geq0} S_k>q\right) 
	= -\inf_{x>0} xI\left(\frac1x\right) 
	&= -\sup(\theta:\Lambda(\theta)\leq 0)
	=: -\loynes.
\end{align}
This tail asymptote, dubbed Loynes' exponent in \cite{Duffy11}, has
practical significance through its interpretation in terms of the
ruin probabilities of an insurance company \cite{Asmussen10} and
in terms of the tail asymptote for the waiting time of a single
server queue \cite{Asmussen03}. A natural estimate \cite{Duffield95b}
of Loynes' exponent is
\begin{align}
\estLo=\sup\{\theta:\estLambda(\theta)\leq 0\}.\label{eq:loynesest}
\end{align}
As an application, results concerning the behavior of the estimates
$\{\estLo\}$ will be established here.

As a second example of why it is valuable to have results on the
stochastic properties of the estimators as random functions rather
than single point values, it has recently been proved \cite{Duffy14}
that the most likely paths to a large integrated random walk with
negative drift \cite{CTCN,DuffyMeyn10, Kulick11, blanchet13} mimic
scaled versions of the function $-\Lambda(\theta)$ for
$\theta\in[0,\delta]$, where $\delta$ is defined in equation
\eqref{eq:loynes}. In order to empirically estimate these nature
of these paths, one must estimate the entire random function $\Lambda$
directly from observations.

In prior work \cite{Duffy05} it was shown that if $X_1$ is bounded
then $\{I_n\}$, considered as a sequence of random lower-semicontinuous
functions, satisfies the LDP in a suitable topological space. The
methods methods there do not generalize to the unbounded random
variable setting, which is often of interest in practice. Motivated
by the estimation of Loynes' exponent in equation \eqref{eq:loynes},
in \cite{Duffy11} it is conjectured that such a generalization is
true. Using a significantly distinct approach from that in
\cite{Duffy05}, here we establish that this is the case for any
distribution of $X_1$ on $ \R$.

\section{A topological setup suitable for the LDP}

Recall that a sequence of random elements $\{Y_n\}$ taking values
in a topological space $(\Y,\tau)$ satisfies the LDP
\cite{Varadhan66,Dembo98} if there exists a lower-semicontinuous
function $\Iy:\Y\mapsto[0,\infty]$ that has compact level sets such
that for all $G$ open and all $F$ closed
\begin{align}
\label{eq:LDP}
-\inf_{y\in G} \Iy(y) \leq
	\liminf_{n\to\infty} \frac 1n \log P(Y_n\in G) 
\text{ and }
\limsup_{n\to\infty} \frac 1n \log P(Y_n\in F) 
	\leq -\inf_{y\in F} \Iy(y).
\end{align}

Considering the sequences of estimators $\{\estM\}$, $\{\estLambda\}$
and $\{\estI\}$ as random lower semi-continuous convex functions,
we prove that they satisfy the LDP in suitable spaces equipped with
appropriate topologies. In particular, we consider
the $\estM$ first as elements of 
\begin{align*}
\MSpacebig = \{f:\R\mapsto[0,\infty] :
f \text{ is a lower-semicontinuous convex function with }
f(0) \text{ finite}\},
\end{align*}
and later as elements of 
\begin{align*}
\MSpacesmall=\{f:\R\mapsto(0,\infty]:\:&f\in\MSpacebig,
	\text{ and there exists a probability measure }\nu\text{ on }\R 
	\text{ such } \\ &\text{that }f(\theta)<\infty \text{ implies } f(\theta) = E_\nu(\exp(\theta x)) \}.
\end{align*}
The $\estLambda$ will be elements of
\begin{align*}
\LSpace = \{f:\R\mapsto(-\infty,\infty] :
f(\theta)=\log g(\theta)\text{ for all } \theta\in\R \text{ and for some }g\in\MSpacesmall\}
\end{align*}
and the $\estI$ be elements of
\begin{align*}
\ISpace = \left\{f:\R\mapsto[0,\infty] :
f(x)=\sup_{\theta\in\R}(\theta x-g(\theta))\text{ for all } x\in\R \text{ and for some }g\in\LSpace
\right\}.
\end{align*}
The space for Jarzysnky estimators, $\{\estJ\}$,
defined in (\ref{eq:estJ}), is a little more complex and will be elements of
\begin{align*}
\JSpace=\left\{f:\R\mapsto[-\infty,\infty] :
f(\theta)=\frac{g(\theta)}{\theta}\text{ for all }\theta\neq 
0\text{ and for some }g\in\LSpace, f(0)=\liminf_{\theta\rightarrow 0}f(\theta)\right\}.
\end{align*}
All of these spaces are equipped with the Attouch-Wets topology
\cite{Attouch83,Attouch86,Beer93}, denoted $\taw$, which is also
known as the bounded-Hausdorff topology, and its Borel $\sigma-$algebra.
This topology was first developed to capture a good notion of
convergence of optimization problems defined through a sequence of
lower-semicontinuous functions. At a fundamental level, its
appropriateness for our needs is demonstrated by the functional
continuity of the Legendre-Fenchel transform, as used in equation
\eqref{eq:estI}, \cite{Beer93}[Theorem 7.2.11].

For proper extended real-valued functions defined over the reals,
the topology is constructed in the following fashion.
Each lower-semicontinuous function $f:\R\to[-\infty,\infty]$ is
uniquely identified with a closed set in $\R^2$, its epigraph
$\epi(f)=\{(\theta,b):b\geq f(\theta)\}$. One then defines convergence of
the functions based on a notion of convergence of closed sets. In
particular, it is based on a projective limits topology using a
bounded-Hausdorff idea: a sequence $\{f_n\}$ of lower-semicontinuous
functions converges to $f$ in $\taw$, if given any bounded set
$B\in\R\times\R$ and any $\epsilon>0$, there exists $N_\epsilon$
such that
\begin{align}
\sup_{x\in B}|d(x,\epi(f_n))-d(x,\epi(f))| <\epsilon \text{ for all }
n>N_\epsilon,
\text{ where } d(x,\epi(f)) = \inf_{y\in\epi(f)} d(x,y)\label{eq:topology}
\end{align}
and we employ the box metric in $\R^2$, $d(x,y)=\max(|x_1-y_1|,|x_2-y_2|)$.

Note that not every element of $\JSpace$ is lower semi-continuous.
Every $f\in\JSpace$ is continuous on the interior of the interval
on which it is finite, and lower semi-continuous on $(0,\infty)$,
as in both cases every CGF is. Therefore the only point at which
$f$ can fail to be lower semi-continuous is
$a=\inf\{\theta:f(\theta)>-\infty\}$ if $a>-\infty$, $a<0$ and
$f(a)>-\infty$. For such functions $f$ we will associate $f$ with
the closure of its epigraph in $\R^2$, or equivalently, with the
epigraph of its lower semi-continuous regularisation,
$\overline{f}(\theta)=\liminf_{\gamma\rightarrow \theta}f(\gamma)$.
In this setting, the symmetric difference of $\epi(f)$ and its
closure is at most one half-line in $\R^2$, and moreover the mapping
from functions in $\JSpace$ to the closure of their epigraphs is
injective, justifying this approach. For this reason, and for ease
of notation, for every $f\in\JSpace$ we will use the notation
$\epi(f)$ to denote the closure of the epigraph of $f$ without
clarification.

As well as the continuity of the Legendre-Fenchel transform, this
topology has many properties that are appropriate for our estimation
problem and that more commonly used function space topologies do not
possess. For example, the following sequence of functions, intended
to be indicative of possible estimates of the cumulant generating
function when $X_1$ has an exponential distribution with rate $1$,
\begin{align*}
f_n(\theta) = 
	\begin{cases}
	e^\theta & \text{ if } \theta \leq 1,\\
	e+n(\theta-1) & \text{ if } \theta > 1,\\
	\end{cases}
\end{align*}
would not be convergent in the topology of uniform convergence on
compacts or the Skorohod topology, but in $\taw$ 
converge to 
\begin{align*}
f(\theta) = 
	\begin{cases}
	e^\theta & \text{ if } \theta \leq 1,\\
	+\infty & \text{ if } \theta > 1,
	\end{cases}
\end{align*}
which is self-evidently desirable. In particular, the topology
captures closeness of functions when their effective domains do
not coincide.

The LDP for each of these collection of estimators is established
in the coming sections. We first prove that the sequence of
maximum likelihood estimators for the MGFs, $\{\estM\}$ defined in
equation \eqref{eq:estM}, satisfy the LDP. From this the LDP for
estimates of the CGF and the rate function shall be obtained by
the contraction principle \cite{Dembo98}.

\section{Statement of Main Results}

In order to establish the results, a substantial volume of work is
necessary, including the characterization of Attouch-Wets limits
of MGFs. The statement of the main results are also a little involved
as it happens that $I_M(f)<\infty$ for some functions $f\in\MSpacebig$
that are not MGFs; such functions are what motivates the definition of $\MSpacesmall$. All proofs are deferred to
Section \ref{sec:proofs}. 

Let the measure on $\R$ corresponding to
$X_1$ be denoted $\mu$ and let $\nu$ be any other measure. We define
the relative entropy, e.g. \cite{Dembo98}, as
\begin{align*}
\H(\nu|\mu) = \begin{cases}
	\displaystyle \int_\R \frac{d\nu}{d\mu}
		\log\left(\frac{d\nu}{d\mu}\right) d\mu
		& \text{ if } \nu << \mu\\	
	+\infty & \text{otherwise},
		\end{cases}
\end{align*}
where $\nu<<\mu$ indicates that $\nu$ is absolutely continuous with
respect to $\mu$ and $d\nu/d\mu$ is the Radon-Nikodym derivative.

\begin{definition}
For any measure $\nu$ we define the MGF associated to $\nu$, $f_\nu\in\MSpacebig$,
by 
\begin{align}
\label{eq:mgfnu}
f_\nu (\theta) = \nu(\exp(\theta\cdot)) = \int_\R e^{\theta x}d\nu.
\end{align}
\end{definition}

\begin{definition}
For any function $f\in\MSpacebig$, let
$\effdomf=\{\theta:f(\theta)<\infty\}$ denote its effective domain
and $\overline{\effdomf}$ the closure of $\effdomf$.
For any $\alpha,\beta$ satisfying $-\infty\leq\alpha\leq 0\leq\beta\leq\infty$ define
$\effdomsub=\{f\in\MSpacebig:\overline{\effdomf}\subset[\alpha,\beta]\}$, the
set of all functions $f\in\MSpacebig$ whose closure of effective domain is a
subset of $[\alpha,\beta]$. Define
$\effdom=\{f\in\MSpacebig:\overline{\effdomf}=[\alpha,\beta]\}$, the
set of all functions $f$ whose closure of effective domain is
$[\alpha,\beta]$.
\end{definition}
When considering $[\alpha,\beta]\subset\R$ we identify
$[\alpha$ with $(\alpha$ when $\alpha=-\infty$, and $\beta]$ with
$\beta)$ when $\beta=\infty$.
\begin{definition}
For any function $f\in\MSpacebig$ that is not a MGF, i.e. for
which there exists no probability measure $\nu$ such that
$f(\theta)=f_\nu(\theta)$ for all $\theta\in\R$, but instead
satisfies $f(\theta)=f_\nu(\theta)$ for some $\nu$ and all $\theta\in\mathcal{D}_f$,
then we say that $f$ mimics $\fmimic$.
\end{definition}
Note that $0\in\effdomf$ for all $f\in\MSpacebig$, so if $f$ mimics
$f_\nu$ then $f(0)=1$. Also, $f\in\MSpacebig$ satisfies $f\in\MSpacesmall$
if and only if $f$ is a MGF or mimics one. Armed with one more set
of definitions, we can state our main results.
\begin{definition}
Define the logarithmic operator $\l:\MSpacesmall\rightarrow\LSpace$
by 
\begin{align}
\label{def:l}
\l(f)(\theta)=\log(f(\theta)) \text{ for all } \theta\in\R
\end{align}
where by convention $\log(\infty)=\infty$. Define
the Jarzynky operator $\j:\LSpace\rightarrow \JSpace$ by 
\begin{align}
\label{def:j}
\j(f)(\theta)=
	\begin{cases}
	f(\theta)/\theta &\theta\neq 0 \\
	\displaystyle\liminf_{\theta\rightarrow 0}f(\theta) &\theta=0.
	\end{cases}
\end{align}
Finally, define the Legendre-Fenchel transform $\lf:\LSpace\rightarrow\ISpace:$ by
\begin{align}
\label{def:lf}
\lf(f)(x)=\sup_{\theta\in\R}(\theta x-f(\theta)),
\end{align}
for $x\in\R$.
\end{definition}
Notice that these functions are bijective, so that their inverse exists. In fact, $\lf$ is an involution \cite{Rockafellar70}[Theorem 26.5], justifying the presentation of the following statements.
\begin{theorem}[LDP for large deviation estimates]\label{thm:MLDP}
If $\{X_n\}$ are i.i.d. real valued random variables, then the
following hold.
\begin{enumerate}
\item
	The sequence of empirical MGF estimators, $\{\estM\}$ defined in
	\eqref{eq:estM}, satisfies the LDP in $\MSpacesmall$ equipped with $\taw$
	and with the convex rate function $\IM$ that possesses the following properties.
	\begin{enumerate}[(a)]
	\item For any $f_\nu$ finite on a non-empty open interval,
	\begin{align*}
	\IM(f_\nu)=\H(\nu|\mu).
	\end{align*}
	\item For $f$ such that $f(0)=1$ and $f(\theta)=\infty$ otherwise, 
	\begin{align*}
	\IM(f)=\inf_{\{\nu:f_\nu=f\}}\H(\nu|\mu). 
	\end{align*}
	\item For any $f\in\effdom$ that is not a MGF, but mimics 
	a MGF $\fmimic$,
	\begin{align*}
	\IM(f)
		&=\begin{cases}
		\IM(\fmimic) 
			&\text{if }\IM(g)<\infty\text{ for some }g\in\effdom,\;g\text{ a MGF} \\
			+\infty &\text{otherwise.}
		\end{cases}
	\end{align*}
	\end{enumerate}
\item
	The sequence of empirical CGF estimators,
	$\{\estLambda\}$, satisfies the LDP in $\LSpace$ equipped with $\taw$ and
	rate function
	\begin{align*}
	\IL(f) = \IM(\l^{-1}(f)) = \{\IM(g): g(\theta)=\exp(f(\theta)) 
		\text{ for all } \theta\in\R\}.
	\end{align*}
	\item 
	The sequence of empirical Jarzynski estimators, $\{\estJ\}$,
	satisfies the LDP in $\JSpace$ equipped with $\taw$ and rate function
	\begin{align*}
	\IJ(f) = \IM(\l^{-1}(\j^{-1}(f))) = \left\{\IM(g): 
		g(\theta)=\exp\left(\theta f(\theta)\right)\text{ for all } \theta\neq 0, g(0)=1\right\}.
	\end{align*}
\item
	The sequence of empirical rate function estimators, $\{\estI\}$,
	satisfies the LDP in $\ISpace$ equipped with $\taw$ and rate function
	\begin{align*}
	\II(f) = \IM(\l^{-1}(\lf^{-1}(f))) = \left\{\IM(g): 
		g(\theta)=\exp\left(\sup_{x\in\R}(x\theta-f(x))\right) 
		\text{ for all } \theta\in\R\right\}.
	\end{align*}
\end{enumerate}
\end{theorem}

The LDPs for the sequences $\{\estLambda\}$, $\{\estJ\}$ and $\{\estI\}$ follow from
that for $\{\estM\}$ by the contraction principle on noting that
the functionals mapping $\estM$ to $\estLambda$, $\estLambda$ to $\estJ$ and $\estLambda$ to $\estI$ are continuous. 
Thus the primary effort in establishing Theorem \ref{thm:MLDP} is
to prove the LDP for the MGF estimators, $\{\estM\}$, which can be
found in Section \ref{MGFEst}, followed by the characterization of
the rate function, which is described in Section \ref{charsec}.

From Theorem \ref{thm:MLDP}, we prove that the estimators $\{\estM\}$,
$\{\estLambda\}$, $\{\estJ\}$ and $\{\estI\}$ converge in probability
to $f_\mu$, $\Lambda_\mu$, $J_\mu$ and $I_\mu$, the MGF, CGF,
effective bandwitch and rate function of the underlying distribution.
This is harder to establish than one might reasonably expect. It
is a well-known result of Large Deviation Theory, e.g. \cite{Lewis95},
that the sequences of probability measures satisfying a LDP with
rate function $I$ are eventually concentrated on the level set
$\{x:I(x)=0\}$, which is compact.  Proving eventual concentration
on smaller sets has also been explored \cite{Lewis95A}. As $\IM$
does not have a unique zero in general, convergence of $\{\estM\}$
to $f_\mu$ in probability is not immediate.  However we do not apply
the results of \cite{Lewis95A} and instead take an alternate
approach to prove the following result.

\begin{corollary}[Weak laws]
\label{cor:weaklaw}
$\{\estM\}$ converges in probability to $f_\mu$, while $\{\estLambda\}$, $\{\estJ\}$ and $\{\estI\}$ converge in probability to $\Lambda_\mu$, $J_\mu$ and $I_\mu$,
respectively.
\end{corollary}

In Section \ref{loynessec}, as an example application of these
results, the following LDP is established for estimates of Loyne's
exponent along with a discussion of some of the properties of its
associated rate function, $\ILo$.

\begin{theorem}[LDP for Loynes' exponent estimates]\label{thm:loynes}
The sequence of Loynes' estimators,
$\{\estLo\}$ defined in (\ref{eq:loynesest}), satisfies a LDP in $[0,\infty]$ with rate function
$\ILo:[0,\infty]\rightarrow[0,\infty]$,
\begin{align*}
\ILo(x)=\inf_{f\in C_x}\IM(f),
\end{align*}
where 
\begin{align*}
C_x&=\{f:f(x)=1,\text{ or }f(x)<1\text{ and }f(y)=\infty\text{ for all } y>x\} \text{ for } x\in(0,\infty), \\
C_0&=\{f:f(x)\geq 1 \text{ for all } x>0\}, \\
 \text{ and } C_\infty&=\{f:f(x)\leq 1\text{ for all } x\geq 0\}.
\end{align*}
\end{theorem}

In the proofs of some of the results in Section \ref{loynessec}
other characterisations of $\ILo$ are considered. The form in
Theorem
\ref{thm:loynes} arises most naturally in our proof of the LDP
and, as discussed in Section \ref{loynessec}, is illustrative of
the discontinuity of the Loynes' exponent mapping.

\section{MGF Estimation}
\label{MGFEst}

For the MGF estimates, we will first establish a LDP in $\MSpacebig$, and then reduce it to a LDP in $\MSpacesmall$. Our method of proof will be first to exclude
functions that could not possibly be close to estimates and then
use the super-additivity methodology pioneered by Ruelle \cite{Ruelle65}
and Lanford \cite{Lanford73}, and elucidated in \cite{Lewis95,Dembo98},
to establish the LDP. Namely, for any open $G\in\taw$, define
\begin{align*}
\munder(G) = \liminf_{n\to\infty} \frac 1n \log P(\estM\in G).
\text{ and }
\mover(G) = \limsup_{n\to\infty} \frac 1n \log P(\estM\in G)
\end{align*}
and their inf-derivatives
\begin{align*}
\inf_{G\ni f} \munder(G) \in[-\infty,0],
\text{ and }
\inf_{G\ni f} \mover(G) \in[-\infty,0] 
\end{align*}
where the infimum is taken over all open sets $G$ containing $f$
or, indeed, $G$ in any local base of the topology around $f$. 
The inf-derivatives are referred to as the lower and upper deviation
functions, respectively, \cite{Lewis95}. When they coincide for all
$f\in\MSpacebig$, they provide the candidate rate function
\begin{align}
\label{eq:weakLDP}
\IM(f) := -\inf_{G\ni f} \munder(G) 
	= -\inf_{G\ni f} \mover(G) \in[0,\infty],
\end{align}
and the sequence $\{\estM\}$ satisfies the weak \cite{Dembo98}, or
vague \cite{Lewis95}, large deviation principle with rate function
$\IM$.  That is, with the LDP upper bound, equation \eqref{eq:LDP},
only holding for all compact sets rather than all closed sets.  The
full LDP, including goodness of the rate function, is then proved
by establishing that exponential tightness holds; i.e. that there is a
sequence of compact sets whose complementary probabilities are
decaying at an arbitrarily high rate.

This super-additivity approach does not provide the characterisation
of $\IM$ described in Theorem \ref{thm:MLDP} and instead that is
developed in Section \ref{charsec}.

\subsection{Reduction of the space}
We wish to show that equation \eqref{eq:weakLDP} holds for all
$f\in\MSpacebig$. We begin this process by eliminating cases where
necessarily $\inf_{G\ni f} \mover(G) =-\infty$. To determine which
functions we must consider, we need to characterise the closure of
the support of $ P(\estM\in\cdot)$, as any functions, $f$, outside
this set will have an open neighbourhood $G$ such that $ P(\estM\in
G)=0$ for all $n$, so that $\inf_{G\ni f}\mover(G)= \inf_{G\ni
f}\munder(G)=-\infty$ and \eqref{eq:weakLDP} holds.

Defining
\begin{align}
\label{def:bspace}
\BSpace=
	\left\{f_\nu:\nu\text{ is compactly supported in }\R\right\} 
	\subseteq \MSpacebig,
\end{align}
we have that $P(\estM\in\BSpace)=1$ for all $n$. To see which
functions lie in its closure $\overline{\BSpace}$, whose complement
forms part of the set of impossible estimates, we establish the
following result.

\begin{proposition}[Characterization of possible limits of MGFs in $\MSpacebig$]
\label{prop:limitsprop}
If $\{f_n\}\subset\BSpace$ is a convergent sequence in $\MSpacebig$
with limit $f\in\MSpacebig$ in $\taw$, then $f$ satisfies one
of the following:
\begin{enumerate}
\item $f$ is a MGF;
\item $f(0)<1$;
\item $f$ mimics a MGF.
\end{enumerate}
\end{proposition}
\begin{proof}[Proof of Proposition \ref{prop:limitsprop}]
See Section \ref{sec:proofs}.
\end{proof}

For the remainder of this section we refer to these classes of
functions as Type 1, 2 and 3 respectively. Establishing the possible
existence of these limits can be done by example.

That Type 1 functions exist is self-evident.  For Type 2 functions,
consider $X_n$ equal to $n$ or 0, each with probability 1/2. Then
the MGF of $X_n$ is $f_n(\theta)=\frac{1}{2}+e^{\theta n}/2$. This
converges in $\taw$ to $f(\theta)=1/2$ for $\theta\leq 0$ and
$f(\theta)=\infty$ otherwise. For Type 3 functions, the function
$f(\theta)=1$ for $\theta\leq 0$ and $f(\theta)=\infty$ otherwise
is seen to be the limit of the functions $f_n(\theta)=(n-1)/n+e^{\theta
n}/n$,
 which are the MGFs of the random
variables $X_n$ that are equal to $n$ with probability $1/n$ and 0
otherwise. The function $f$ mimics the MGF of the weak limit of
$\{X_n\}$, however $f$ itself is not a MGF. Showing that all limits 
of $\{f_n\}$ are in one of these classes
is a bigger task for which we adopt a common tactic when considering
the limits of MGFs: an application of the Helly Selection Principle,
e.g. \cite{Natanson55}, to the corresponding sequence of distribution
functions.

It is worth noting that every MGF is in the closure of $\BSpace$,
defined in equation \eqref{def:bspace}. If we take any distribution
$\nu$, and let $\nu_n(dx)=\nu(dx)/\nu([-n,n])$ be $\nu$ conditioned
on $[-n,n]$, then $f_{\nu_n}\in\BSpace$ for all $n$ and
$f_{\nu_n}\rightarrow f_\nu$ in $\taw$. Choosing to let $\BSpace$
only contain MGFs whose distributions have compact support is done
only to ensure every element of $\BSpace$ is finite everywhere,
which simplifies the proof of Proposition \ref{prop:limitsprop}.

\subsection{A local convex base}

Let $B_r=\{x:d(0,x)<r\}\in\R^2$ be the open ball of radius $r$ in
$\R^2$ and $\Bc_r = \{x:d(0,x)\leq r\}\in\R^2$ be its closure. 
To use the super-additivity approach to establish \eqref{eq:weakLDP}
we need to construct a
local convex base for the topology $\taw$, but this is not possible in
general. 
The
following collection of sets
\begin{align} \label{eq:beerbase}
V_k(f) = 
\left\{g:\sup_{x\in \Bc_k}|d(x,\epi(g))-d(x,\epi(f))|<\frac 1k\right\},
	\text{ for } k\in\N,
\end{align} 
is known to form a local base for the Attouch-Wets topology
\cite{Beer93}. The sets $V_k(f)$ defined in equation \eqref{eq:beerbase}
are not, however, typically convex in the sense that if $g,h\in
V_k(f)$ then we cannot deduce that $l_\alpha$, defined by
$l_\alpha(\theta)=\alpha g(\theta)+(1-\alpha)h(\theta)$ for all
$\theta\in\R$ and for every $\alpha\in(0,1)$, is in $V_k(f)$. As a
counter example, graphically illustrated in Figure \ref{fig:nonconvex},
consider $V_2(g)$ where, for any $\beta>2$,
\begin{align*}
g(\theta) = 
	\begin{cases}
	0 & \text{ if } \theta = 0,\\
	\infty & \text{ if } \theta \neq 0,
	\end{cases}
\text{  }
h(\theta)=
	\begin{cases}
	1-\beta\theta  & \text{ if } \theta \in[0,1/\beta],\\
	\infty & \text{ if } \theta\notin[0,1/\beta],
	\end{cases}
\text{ and thus }
l_{1/2}(\theta) = 
	\begin{cases}
	1/2 & \text{ if } \theta = 0,\\
	\infty & \text{ if } \theta \neq 0,
	\end{cases}
\end{align*}
so that  
\begin{align*}
\sup_{x\in \Bc_2}|d(x,\epi(g))-d(x,\epi(h))|=1/\beta,
\text{ but }
\sup_{x\in \Bc_2}|d(x,\epi(g))-d(x,\epi(l_{1/2}))|=1/2,
\end{align*}
so that $g,h\in V_2(g)$, but $l_{1/2}\notin V_2(g)$. 

\begin{figure}[h]
\begin{center}
\includegraphics[scale=0.6, trim=0 0 0 0]{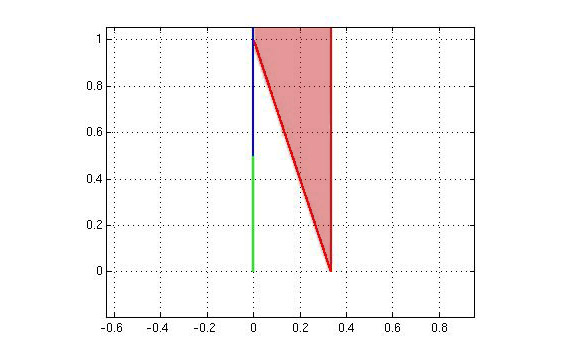}
\caption{The usual local base for $\taw$ is not convex. For $\beta=3$
and with the functions defined in the text,
this is illustrated here for $V_2(g)$, where $g$ (green) and
$h$ (red) are in $V_2(g)$, but their linear combination $l_{1/2}$ (blue)
is not.}
\label{fig:nonconvex}
\end{center}
\end{figure}

Although the base $\{V_k(f)\}$ is not convex, it will suffice for
our initial elimination of impossible MGF estimates. For the core
result, we introduce a new base that is convex where it matters;
that is, at functions that could appear as limits of MGF estimates.

As convergence in $\taw$ does not imply point-wise convergence
\cite{Beer93}, even though $\estM(0)=1$ for all $n$ it is possible
that $\estM(0)$ converges to a value less than $1$ in $\taw$. As
a direct example, consider the
following sequence:
\begin{align*}
f_n(\theta) = 
	\begin{cases}
	n\theta+1 & \text{ if } |\theta|\leq 1/n,\\
	+\infty & \text{ if } |\theta|>1/n.
	\end{cases}
\end{align*}
Point-wise we have that $f_n(0)$ converges to $1$, but $\epi(f_n)$
converges in $\taw$ to the epigraph of the function that is $0$ at
$0$ and $+\infty$ elsewhere. This is illustrated in Figure \ref{fig:f0}
 and occurs as the topology of point-wise convergence is neither
stronger nor weaker than $\taw$ \cite{Beer93}. 

As a result, we must include in our considerations functions for
which $f(0)<1$, but it is not always the case that a local convex
base exists for these functions. To see this consider $f$ satisfying
$f(0)<1$ and $f(\theta)=\infty$ for $\theta\neq 0$. Assume $C_k(f)$
forms a local convex base for $f$, and consider $k$ so that $C_k(f)$
does not contain the function $g$ satisfying $g(0)=1$ and
$g(\theta)=\infty$ for $\theta\neq 0$. Then take two 
functions in $C_k(f)$, one infinite on the right half-plane,
and one infinite on the left half-plane; any non-trivial convex
combination of them will equal $g$. 

Similarly, when constructing the convex base we must rely on functions
that satisfy $f(0)>1$, which is why we included them in $\MSpacebig$.
Despite these issues, the following result shows directly that the
rate function evaluated at these functions is $+\infty$.

\begin{proposition}[Functions with infinite rate]\label{prop:Lessthan1}
For any $f\in\MSpacebig$ such that $f(0)\neq 1$, 
\begin{align*}
\inf_{G\ni f}\mover(G)=-\infty.
\end{align*}
\end{proposition}
\begin{proof}
See Section \ref{sec:proofs}.
\end{proof}

\begin{figure}[h]
\centering
\includegraphics[scale=0.6]{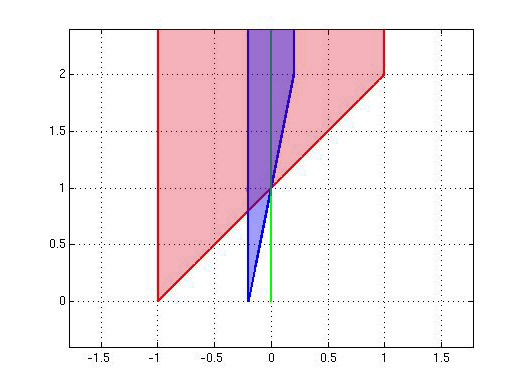}
\caption{$f_1$ in red, $f_5$ in blue, and $f$ in green.}
\label{fig:f0}
\end{figure}

For the function satisfying $f(0)=1$ and $f(x)=\infty$ for $x\neq0$,
we rely on the idea tha $V_k(f)\cap\BSpace$ is convex to prove subadditivity. For the
general class of functions satisfying $f(0)=1$ and finite on a
non-empty open interval, we can construct a local convex base
$\{A_k(f)\}$ such that
\begin{align}
\label{eq:RLf}
\mover(A_k(f))=\munder(A_k(f)).
\end{align}
This will enable us to deduce the weak \cite{Dembo98} (or vague
\cite{Lewis95}) LDP from which the LDP follows by Proposition
\ref{prop:Mcomp}.

The idea for the base $A_k(f)$ is to consider a small vertical shift
of $f$, decrease its effective domain slightly and then intersect
an element of the non-convex base defined in equation \eqref{eq:beerbase}
around this new function with all the functions whose epigraph
strictly contains the resulting curtailed function's epigraph.
We begin this process by defining the shifted and curtailed functions
$\{\fs_k\}$ from which the base will be built. 
\begin{definition}
For each $f\in\MSpacebig$ such that $f(0)=1$ and $\effdomf\neq\{0\}$,
and for each $k\in\N$, let 
\begin{align}
\label{eq:fB}
\fBl = \inf\{\theta:(\theta,f(\theta))\in B_{2k+2}\}
\text{ and }
\fBr = \sup\{\theta:(\theta,f(\theta))\in B_{2k+2}\}.
\end{align}
Let $0<\epsilon<(\fBr-\fBl)/2$ be such that
\begin{align}
d((\alpha,f(\alpha)),(\beta,f(\beta))) &< \frac{1}{2k}
	\text{ for all } \alpha,\beta\in [\fBl,\fBl+\epsilon] \\
\text{ and }
d((\alpha,f(\alpha)),(\beta,f(\beta))) &< \frac{1}{2k}
	\text{ for all } \alpha,\beta\in [\fBr-\epsilon, \fBr].
	\label{eq:overhangs}
\end{align}
Then we define $\fs_k$ by
\begin{align*}
\fs_k(\theta) =
	\begin{cases}
	\displaystyle f(\theta) +\frac{1}{2k} 
		& \text{ if } \theta\in [\fBl+\epsilon,\fBr-\epsilon] \\ 
	+\infty & \text{ otherwise.}
	\end{cases}
\end{align*}
\end{definition}
This curtailing process is illustrated in Figure \ref{fig:fsharp}.

\begin{figure}[h]
\includegraphics[scale=0.4]{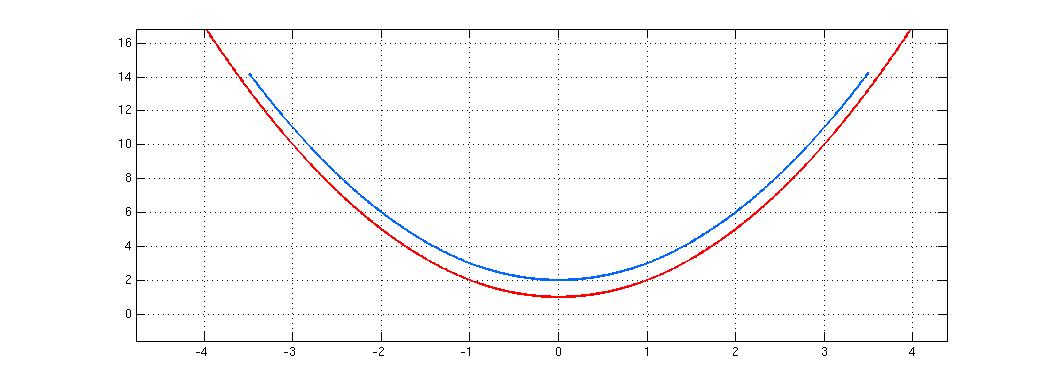}
\caption{Illustration of the curtailing of $f(x)=x^2+1$ (red) to
$\fs_7$ (blue), whose shrunk domain is
$\eta_{f,r,7}=\sqrt{15}=-\eta_{f,l,7}$. Both the $\epsilon$ curtailment
and the $1/2k$ vertical shift of $\fs_7$ are exaggerated for
illustrative purposes.}
\label{fig:fsharp}
\end{figure}

Using these shifted functions, we construct a
collection of sets that we shall prove form a local, convex base
at $f$. In order to do so, we require the following piece of notation.

\begin{definition}
For each $f\in\MSpacebig$ such that $f(0)=1$ and $\effdomf\neq\{0\}$,
define
\begin{align*}
\fl = \inf\effdomf\text{ and }\fr = \sup\effdomf.
\end{align*}
For two functions
$f$ and $g$, we write $g\lll f$ if $\gl<\fl$, $\fr<\gr$, and $g(\theta)<f(\theta)$ for all $\theta\in[\fl,\fr]$. 
\end{definition}
Notice that equivalently we can say that $g\lll f$ if $\epi(f)\subset\epi(g)$,
$\fl$ and $\fr$ are finite and 
\begin{align*}
\inf_{x\in\partial\epi(f),y\in\partial\epi(g)}d(x,y)>0, 
\end{align*}
where $\partial$ denotes the boundary of a set, which is the property
that motivates this definition.
\begin{figure}[h]
\center
\includegraphics[scale=0.4]{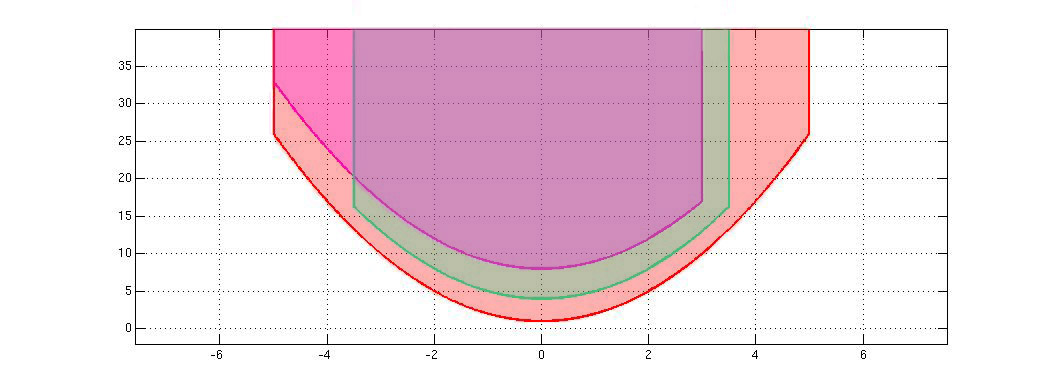}
\caption{$f$, $g$ and $h$ in red, purple and green respectively. $h\lll f$, but $g\lll\llap{/\;\;\:} f$, as $\gl=\fl$.}
\end{figure} 
\begin{definition}
For each $f$ such that $f(0)=1$ and $\effdomf\neq\{0\}$ and each $k\in\N$,
define the set
\begin{align*}
A_k(f) = V_k(\fs_k)\cap W(\fs_k),
\end{align*}
where $V_k$ is defined in equation \eqref{eq:beerbase} and $W(\fs_k)
= \{g: g \lll \fs_k\}$. 
\end{definition}

\begin{proposition}[Local convex base]
\label{prop:Abase}
For each $f$ such that $f(0)=1$ and $\effdomf\neq\{0\}$, $\{A_k(f)\}$ forms a
local convex base at $f$.
\end{proposition}
\begin{proof}
See Section \ref{sec:proofs}.  
\end{proof}

\subsection{Coincidence of the deviation functions, exponential tightness and the LDP}

Using the new base, we can prove the following result, following
the super-additivity method of Ruelle and Lanford, to establish
Cram\'er's Theorem.

\begin{proposition}[Super-additivity]
\label{prop:super}
If $f$ is such that $f(0)=1$
and $\effdomf\neq\{0\}$ then for each $k\in\N$,
\begin{align*}
\mover(A_k(f))=\munder(A_k(f)).
\end{align*}
That is, equation \eqref{eq:RLf} holds. If $f(0)=1$ and $f(\theta)=+\infty$ for $\theta\neq 0$,
\begin{align*}
\mover(V_k(f))=\munder(V_k(f)).
\end{align*}
\end{proposition}
\begin{proof}
See Section \ref{sec:proofs}.
\end{proof}

In our setting, exponential tightness for
$\{\estM\}$ will prove to be near automatic due to the following
proposition.

\begin{proposition}[Compactness]
\label{prop:Mcomp}
The set $\{f\in\MSpacebig:f(0)\leq 1\}$ is compact.
\end{proposition}
\begin{proof}
See Section \ref{sec:proofs}.
\end{proof}
As $P(\estM(0)>1)=0$, exponential tightness is immediate. 

A combination of these results leads us to the LDP for $\{\estM\}$,
albeit without a good characterization of the rate function.

\begin{theorem}[LDP for MGF estimators]
\label{thm:MLDPnonchar}
The sequence of empirical MGF estimates, $\{\estM\}$, satisfies the LDP in 
$\MSpacesmall=\overline{\BSpace}\cap\{f:f(0)=1\}\subseteq\MSpacebig$ equipped
with $\taw$ and rate function $\IM:\MSpacesmall\mapsto[0,\infty]$,
\begin{align*}
\IM(f) =-\inf_k \mover(A_k(f))=-\inf_k \munder(A_k(f)).
\end{align*}
\end{theorem}
\begin{proof}
See Section \ref{sec:proofs}.
\end{proof}

Although as yet we do not have a good characterisation of $\IM$,
from this result proving the LDP for the MLEs of the MGF as random
functions in the Attouch-Wets topology, we can establish the LDP
for the CGF estimates and the rate function estimates via somewhat
involved applications of the contraction principle. We use the
contraction principle with the map $\l$ defined in (\ref{def:l}) to prove the LDP for
the CGF estimators, $\{\estLambda\}$.

\begin{lemma}[Continuity of $\l$]
\label{lem:Lcont}
The functional $\l$, defined in (\ref{def:l}) is continuous.
\end{lemma}
\begin{proof}
See Section \ref{sec:proofs}.
\end{proof}

Using this continuity, we are in a position to prove the result
for the CGF estimates.

\begin{theorem}[LDP for CGF estimators]
\label{thm:LLDP}
The sequence of empirical cumulant generating function estimators,
$\{\estLambda\}$, satisfies the LDP in $\LSpace$ equipped with $\taw$ and
rate function
\begin{align*}
\IL(f) = \IM(\l^{-1}(f)) = \left\{\IM(g): g(\theta)=\exp(f(\theta))
        \text{ for all } \theta\in\R\right\}.
\end{align*}
\end{theorem}

\begin{proof}
See Section \ref{sec:proofs}.
\end{proof}

Similarly, we can prove the LDP for the Jarzynksy estimators, $\{\estJ\}$ by first 
establishing the continuity of the map $\j$ defined in (\ref{def:j}).

\begin{lemma}[Continuity of $\j$]
\label{lem:JCont}
The functional $\j$, defined in (\ref{def:j}) is continuous.
\end{lemma}

\begin{proof}
See Section \ref{sec:proofs}.
\end{proof}

\begin{theorem}[LDP for Jarzynski estimators]
\label{thm:JLDP}
The sequence of empirical Jarzynski estimators, $\{\estJ\}$, satisfies the LDP in $\JSpace$ equipped with $\taw$ and rate function 
\begin{align*}
\IJ(f) = \IM(\l^{-1}(\j^{-1})(f)) = \{\IM(g): g(\theta)=\exp(\theta f(\theta)) 
		\text{ for all } \theta\neq 0, g(0)=1\}.
\end{align*}
\end{theorem}

\begin{proof}
See Section \ref{sec:proofs}.
\end{proof}

Having established the LDP for the CGF estimates in Theorem
\ref{thm:LLDP}, the rate function estimator result follows from
another application of the contraction principle in conjunction
with the continuity of
 the Legendre-Fenchel transform defined in
equation \eqref{def:lf}. Considering $\lf:\LSpace\mapsto\ISpace$,
the map $\lf$ is a homeomorphism \cite{Attouch86,Beer93} and, indeed,
this is in part what leads us to this topology in order to establish
these results; it correctly captures smoothness in convex conjugation.

\begin{theorem}[LDP for rate function estimators]
\label{thm:ILDP}
The sequence of empirical rate function estimators, $\{\estI\}$,
satisfies the LDP in $\ISpace$ equipped with $\taw$ and rate function
\begin{align*}
\II(f) = \IM(\l^{-1}(\lf^{-1}(f))) = \left\{\IM(g):
        g(\theta)=\exp\left(\sup_{x\in\R}\left(x\theta-f(x)\right)\right)
        \text{ for all } \theta\in\R\right\}.
\end{align*}
\end{theorem}

\begin{proof}
See Section \ref{sec:proofs}.
\end{proof}

\section{Characterizing the Rate Function}\label{charsec} 

The super-additivity approach does not directly provide a useful
form for the resulting rate functions governing the LDPs. As $\l$
and $\lf$ are injective, to characterise $\II$ and $\IL$ it suffices
to characterise $\IM$. Based on the appearance of relative entropy
in the more restricted results in \cite{Duffy05}, one anticipates
it has a role to play here.

The approach we take to create a useful characterization relies
heavily on continuity and inverse continuity of mappings from subsets
of the set of MGFs to measures on $\R$. This
reliance on continuity suggests that there may be some version of
the contraction principle that could be applied directly to
prove the LDP. However, since the mapping from $\MSpacesmall$ to $\Measures$
is not well-defined (consider the function finite only at 0), and
the inverse mapping is not surjective (there are functions in the
effective domain of $\IM$ that are not mapped to by any measure)
this seems unlikely. 

Garcia's extension of the contraction principle, \cite{Garcia04}[Theorem
1.1], almost suffices if we consider Sanov's Theorem for empirical
measures in the weak toplogy \cite{Dembo98} and the map
taking measures to their MGFs. However, due to the possible
unboundedness of the support of those measures, that map is not
continuous for any $x\in\Measures$ and in any case $\IM$ will turn
out not to coincide with the rate function given in \cite{Garcia04};
if $f$ mimics $f_\nu$, we may get $\IM(f)=\infty$ even though, if
it were possible, an application of the contraction principle would
give a finite value. Moreover, there are conditions independent
of the continuity that characterise $\IM$, so it appears that an
alternate approach is necessary.

Propositions \ref{lemma:convex}, \ref{thm:I=H} and \ref{thm:Iformimics},
which follow, together provide the characterization in Theorem \ref{thm:MLDP}.

\subsection{Convexity of $\IM$}

In the process of proving that $\IM$ is convex, we must establish
that $\{(f,g):(f+g)/2\in G\}\subset\MSpacesmall\times\MSpacesmall$ is an
open set in the product topology for each open $G\in\MSpacesmall$. That
is, that averaging is a continuous operation in $\MSpacesmall$.
The proof of this result indicates why convex combinations are
not a continuous operation in $\MSpacebig$, hence the need to restrict
our LDP to $\MSpacesmall$ in Theorem \ref{thm:MLDPnonchar}. As an example, 
consider $f_n(\theta)=1+\theta n$ and $g_n(\theta)=1-n\theta$ for $|\theta|\leq 1/n$
and infinite otherwise,
converging to $f$ and $g$ finite only at 0, where they both equal 0. $(f_n+g_n)/2$ is equal
to 1 on $[-1/n,1/n]$ and infinite otherwise, so that it converges to $(f+g)/2+1$, 
not $(f+g)/2$.

We establish global convexity of $\IM$ by creating an argument along
the lines of \cite{Dembo98}[Lemma 4.1.21]. The conditions of that
Lemma as stated do not hold here as $\MSpacesmall$ is not a topological
vector space: we do not have additive inverses, closure under scalar
multiplication or addition, and nor is there a zero element.  The
proof of \cite{Dembo98}[Lemma 4.1.21], however, relies only on 
two deductive conclusions of those hypotheses, which we establish
directly in Section \ref{sec:proofs}.

\begin{proposition}[Convexity of $\IM$]
\label{lemma:convex}
The MGF rate function, $\IM$, is convex on its entire domain. 
\end{proposition}

\begin{proof}
See Section \ref{sec:proofs}.
\end{proof}

\subsection{Characterizing $\IM$ for MGFs}

\begin{proposition}[$\IM(f)$ for $f$ an MGF]
\label{thm:I=H}
For all moment generating functions $f_\nu$ finite on a non-empty open interval, 
\begin{align*}
\IM(f_\nu)=\H(\nu|\mu).
\end{align*}
Moreover, for $f\in\MSpacesmall$ finite only at 0, 
\begin{align*}
\IM(f)\leq\inf_{\nu:f=f_\nu}\H(\nu|\mu).
\end{align*} 
\end{proposition} 
\begin{proof}
See Section \ref{sec:proofs}.
\end{proof}
This result is proved by three smaller results, the first two of
which rely on continuity or inverse continuity of the $\nu\mapsto
f_\nu$ operation when restricted to certain subsets of $\Measures$
or $\MSpacesmall$. First, for any moment generating function $f_\nu$
finite on a non-empty open interval, $\IM(f_\nu)\geq \H(\nu|\mu)$.
Second, for any $f_\nu\in\BSpace$, $\IM(f_\nu)\leq \H(\nu|\mu)$.
Finally, for any moment generating function $f$, $\IM(f)\leq
\inf_{\left\{\nu:f=f_\nu\right\}}\H(\nu|\mu)$.  For any MGFs $f$
finite on a non-empty open interval, $f=f_\nu$ for exactly one
measure $\nu$, giving us $\IM(f_\nu)=\H(\nu|\mu)$. For $f$ finite
only at 0, more work is required.

\subsection{Characterizing $\IM$ for the MGF finite only at 0 and for non-MGFs}

In short, we prove the following result, contained in Theorem \ref{thm:MLDP}.
\begin{proposition}[$I_M(f)$ for $f$ mimicking a MGF]\label{thm:Iformimics}
$\:$ \\
\begin{enumerate}[(a)]
\item For $f$ satisfying $f(0)=1$ and $f(\theta)=\infty$ otherwise, 
\begin{align*}
\IM(f)=\inf_{\{\nu:f=f_\nu\}}\H(\nu|\mu). 
\end{align*}
\item For any $f\in\effdom$ that is not a MGF but instead mimics the MGF $\fmimic$,
\begin{align*}
\IM(f)&=\IM(\fmimic)+\inf_{\{g\in \effdom,g\text{ a MGF}\}}\IM(g) \\
&=\left\{\begin{array}{ll} \IM(\fmimic) &\text{if }\IM(g)<\infty\text{ for some }g\in\effdom,g\text{ a MGF} \\
\infty &\text{otherwise}.\end{array}\right.
\end{align*}
\end{enumerate}
\end{proposition}
\begin{proof}
Proposition \ref{thm:Iformimics} will follow from Lemmas \ref{lemma:mimics} and \ref{thm:5equiv}
below; see Section \ref{sec:proofs}.
\end{proof}

Part (a) is, perhaps, to be expected in light of earlier results.
The first equality in part (b) appears less useful than the second,
but it gives more insight into the reason that the rate function
takes the form that it does. It may seem surprising that functions
in $\MSpacesmall$ that are not MGFs lie in the effective domain of
$\IM$, but this is the case.

We begin with a condition for when a function $f\in\effdom$ that is
not moment generating function is in the effective domain of $\IM$,
i.e. that any moment generating function in $\effdom$ is in the
effective domain of $\IM$, along with the moment generating function
that $f$ is mimicking. 

\begin{lemma}[Characterizations for $f$ not a MGF]
\label{lemma:mimics}
We have the following characterizations.
\begin{enumerate}
\item
\label{lemma:valuesforspike}
For $f$ satisfying $f(0)=1$ and $f(\theta)=\infty$ for $\theta\neq 0$, $\IM(f)\in\{0,\infty\}$.
\item
\label{lemma:finitemimic}
If $\IM(g)<\infty$ for some function $g\in\effdom$,
$[\alpha,\beta]\neq[0,0]$, then $\IM(f)\leq\IM(\fmimic)$ for
$f\in\effdom$ mimicking $\fmimic$.
\item
\label{lemma:valuesformimics}
If $f\in\effdom$ mimics $\fmimic$ for some $[\alpha,\beta]\neq[0,0]$ then $\IM(f)\in\{\IM(\fmimic),\infty\}$. 
\end{enumerate}
\end{lemma}
\begin{proof}
See Section \ref{sec:proofs}.
\end{proof}

Combining statements \ref{lemma:finitemimic} and
\ref{lemma:valuesformimics} gives us that if $f_\nu$, $f$ and $g$
satisfy $\IM(f_\nu),\IM(g)<\infty$, $g,f\in\effdom$, and $f$ mimics
$f_\nu$, then $\IM(f)<\infty$.  Establishing conditions for when
$\IM(f)=\infty$ is quite involved and to do so we prove the following.

\begin{lemma}[Five equivalences]\label{thm:5equiv}
The following are equivalent for all pairs $(\alpha,\beta)\in[-\infty,0]\times[0,\infty]$:
\begin{enumerate}
\item The random walk associated with $e^{\gamma X_1}$, where the distribution of $X_1$ 
corresponds to the measure $\mu$, satisfies
Cram\'er's Theorem for some $\gamma\not\in[\alpha,\beta]$,
\item $\IM(f)=\infty$ for all $f\in\effdomsub$,
\item $\H(\nu|\mu)=\infty$ for all $\nu$ with $f_\nu\in\effdomsub$,
\item $\H(\nu|\mu)=\infty$ for all $\nu$ with $f_\nu\in\effdom$,
\item $\IM(f)=\infty$ for all $f\in\effdom$.
\end{enumerate}
\end{lemma}
\begin{proof}
See Section \ref{sec:proofs}.
\end{proof}

Note that, in light of Proposition \ref{thm:I=H}, when
$[\alpha,\beta]\neq[0,0]$, (4) can be restated as $\IM(f_\nu)=\infty$
for all $f_\nu\in\effdom$. This provides a sufficient condition for
when $\IM(f)=\infty$ for $f\in\effdom$, $f$ not a MGF, as $(4)\Rightarrow
(5)$ establishes that $\IM(f)=\infty$ for all $f\in\effdom$ if
$\IM(f)=\infty$ for all MGFs in $\effdom$.

Some of these statements may not be of interest in their own right.
The condition (1) is included mainly as an aid to establishing the
other equivalences. Indeed, in light of Proposition \ref{thm:I=H}
we only need $(4)\Rightarrow (5)$ to prove Proposition
\ref{thm:Iformimics}. However, it will be seen in Section
\ref{sec:Examples} that $(1)\Rightarrow(4)$ is useful for characterising
$\IM$ for specific distributions. The relation $(5)\Rightarrow
(1)$ is also of interest, as it gives us the following corollary.
Although it will not be used to prove any of the main results, it
is useful in practical characterisations of $\IM$ for specific
distributions $\mu$, as seen is Section \ref{sec:Examples}.

\begin{corollary}[To lemma \ref{thm:5equiv}]\label{cor:mineffdom}
Under the assumptions of the Theorem \ref{thm:MLDP},\\
\begin{enumerate}
\item There exist $\alpha_0\in[-\infty,0]$, $\beta_0\in[0,\infty]$
such that $f\in\effdom$ mimicking $f_\nu$ with $\IM(f_\nu)<\infty$ satisfies $\IM(f)=\IM(f_\nu)$
if $[\alpha_0,\beta_0]\subset[\alpha,\beta]$, otherwise $\IM(f)=\infty$.
\item For any $f\in\MSpacesmall$, $\IM(f)<\infty\Rightarrow \overline{\effdomf}\supset[\alpha_0,\beta_0]$.
\item For $f\in\MSpacesmall$ finite only at 0, $\IM(f)=0$ if and only if $\alpha_0=\beta_0=0$.
\end{enumerate}
\end{corollary}
\begin{proof}
See Section \ref{sec:proofs}.
\end{proof}
The first statement entirely characterises $\IM$ for functions $f$
mimicking $f_\nu$, saying that if the closure of the effective
domain of $f$ contains some interval then $\IM(f)=\IM(f_\nu)$,
otherwise $\IM(f)=\infty$. It also tells us that $\IM$ has a unique
zero if and only if $[\alpha_0,\beta_0]\ni\mathcal{D}_{f_\mu}$. The
second states that this same interval must be in the closure of the
effective domain of any function $g$ such that $\IM(g)<\infty$, giving
a necessary condition for functions to be in the effective domain
of $\IM$.

\subsection{Proving the main results}
Equipped with these results we are now able to prove Proposition
\ref{thm:Iformimics}. Using Propositions \ref{thm:I=H} and
\ref{thm:Iformimics} and Theorem \ref{thm:MLDPnonchar}, as well as
Theorems \ref{thm:LLDP}, \ref{thm:JLDP} and \ref{thm:ILDP} we can prove Theorem
\ref{thm:MLDP}, from which Corollary \ref{cor:weaklaw} follows. The
proofs appear subsequently in Section \ref{sec:proofs}. Theorem
\ref{thm:loynes} is established in Section \ref{loynessec}.

\section{Examples}\label{sec:Examples}
As summarised in Theorem 3.1, $\IM(f_\nu)=\H(\nu|\mu)$ for all MGFs,
$f_\nu$, finite on a non-empty open interval, while $\IM(f)=\infty$
if $f\in\MSpacesmall$ is neither a MGF nor mimics one. The question
remains as to the value of $\IM(f)$ 
for $f$ mimicking a moment
generating function, $f_\nu$, and for the special function
\begin{align*}
f(\theta)=
\begin{cases} 
	1 & \text{ if } \theta=0\\
	\infty & \text{ otherwise}. 
\end{cases}
\end{align*}
This issue is reduced to that of finding $\alpha_0$ and $\beta_0$
defined in Corollary \ref{cor:mineffdom}.

The following table gives the values of $\alpha_0$ and $\beta_0$
in a variety of contexts. The calculations are straight forward and
are not shown. They can be evaluated by considering for what values
of $\gamma$ does $e^{\gamma X_1}$ have a MGF finite in a neighbourhood
of the origin, giving upper and lower bounds for both $\alpha_0$
and $\beta_0$ respectively, and by showing for what values of
$\alpha$ and $\beta$ Lemma \ref{thm:5equiv} (4) does not hold,
implying that Lemma \ref{thm:5equiv} (1) does not hold and hence giving 
lower and upper bounds for $\alpha_0$ and $\beta_0$ respectively.

\begin{figure}[h!]
\centering
\begin{tabular}{|c| c| c|} \hline 
Distribution & $\alpha_0$ & $\beta_0$  \\ \hline
$\mu$ compactly supported & $-\infty$ & $\infty$ \\ \hline
$\mu\sim \text{Normal}(\eta,\sigma^2)$ & 0 & 0 \\ \hline
$\mu(dx)\propto e^{-e^{\lambda x}}dx$ for $x>0$, $\lambda>0$ & $-\infty$ & $\lambda$ \\ \hline
$\mu(dx)\propto e^{-e^{x^\lambda}}dx$ for $x>0$, $\lambda>1$ & $-\infty$ & $\infty$ \\ \hline
\end{tabular}
\end{figure}

For $\mu$ compactly supported, $\alpha_0=-\infty$ and $\beta_0=\infty$
so that $\IM$ is finite only at MGFs, recovering \cite{Duffy05}[Theorem
1]. As the value of $\beta_0$ depends only on the right tail, and
a heavier right tail in $\mu$ implies a smaller or equal $\beta_0$,
and similar for left tails and $\alpha_0$, the values of
$\alpha_0,\beta_0$ can be gleaned for many other distributions using
the results above. For example, if $\mu\sim\text{Exp}(\lambda)$
then its left tail behaves like a bounded random variable and so
$\alpha_0=-\infty$, and its right tail is heavier than that of a
Normal random variable, and so $\beta_0=0$. The final two distributions
serve as examples of unbounded distributions for which
$\beta_0\in(0,\infty]$. Similar constructions yield
$\alpha_0\in[-\infty,0)$ and so using a combination of these
distributions with others, any pair $(\alpha_0,\beta_0)$ is attained
by some distribution $\mu$.

\section{Estimating Loynes' Exponent}\label{loynessec}
\subsection{A LDP for $\{\estLo\}$}
Consider Loynes' exponent, \eqref{eq:loynes}:
\begin{align*}
\loynes=\sup\{\theta:\Lambda(\theta)\leq 0\}=\sup\{\theta:M(\theta)\leq 1\}\in[0,\infty].
\end{align*}
If we decide to estimate Loynes' exponent, we might do so using $\estM$ in the following way:
\begin{align*}
\loynes_n=\sup\{\theta:\estM(\theta)\leq 1\}.
\end{align*}
Although this mapping from $\MSpacesmall$ to $[0,\infty]$ is not
continuous (consider the sequence of MGFs $f_n(\theta)=e^{\theta/n}$),
we can still use the LDP of $\{\estM\}$ to construct a LDP for
$\{\estLo\}$ with rate function $\ILo$ using Garcia's extension of
the contraction principle \cite{Garcia04}. Assuming $P(X_1\in(-r,r))=0$
for some $r>0$, as in \cite{Duffy11}, makes proof immediate by an
application of Puhalskii's extension of the contraction principle
\cite{Puhalskii91}[Theorem 2.2]. Here we will deal with the general
case. Throughout this section we shall assume that $P(X_1>0)$ and
$P(X_1<0)\in(0,1)$, as otherwise the proof of the LDP is trivial,
and the characterization of the rate function is as stated in Theorem
\ref{thm:loynes}. We also let $x$ be the argument of the functions
in $\MSpacesmall$ so as not to confuse $\loynes$ and $\estLo$ with
$\theta$.

The LDP for Loynes' exponent is proved directly using Garcia's
extension of the contraction principle \cite{Garcia04}, the relevant
part of which is recapitulated below.
\begin{theorem}[Garcia \cite{Garcia04}, Theorem 1]\label{Garcia}
Assume $\Omega\xrightarrow{X_n}\mathcal{X}\xrightarrow{G}\mathcal{Y}$, $\mathcal{X}$, $\mathcal{Y}$ are metric spaces, and $\{X_n\}$ satisfies the large deviation principle with good rate function $I^\#$. Define $G^x=\{y\in\mathcal{Y}|(\:\exists\: x_n\rightarrow x)G(x_n)\rightarrow y)\}$. If for every $x$ with $I^\#(x)<\infty$, $G^x$ satisfies 
\begin{enumerate}
\item Every sequence converging to $x$ has a subsequence along which the function $G$ converges.
\item For every $y\in G^x$ there is a sequence $\{x_n\}$ converging to $x$ such that $G(x_n)\rightarrow y$, $G$ is continuous at $x_n$ and $I^\#(x_n)\rightarrow I^\#(x)$.
\end{enumerate}
Then $\{G(X_n)\}$ satisfies the large deviation principle with good rate function $I$ defined by 
\begin{align*}
I(y)=\inf\{I^\#(x):y\in G^x\}.
\end{align*}
\end{theorem}
Here, $G$ is the Loynes' exponent mapping, $\mathcal{X}=\MSpacesmall$
and $\mathcal{Y}=[0,\infty]$. With this result we are ready to prove
Theorem \ref{thm:loynes}; see Section \ref{sec:proofs}.

Note in the statement of Theorem \ref{thm:loynes} that $C_x$ is
closed (and therefore compact in $\MSpacebig$) for all $x\in[0,\infty]$.
Indeed, for each $x\in[0,\infty]$, $C_x$ is the closure of the set
$\{f:\sup\{y:f(y)\leq 1\}=x\}$, inverse image of $x$ under the
Loynes' exponent mapping. For $x=\infty$ this set is closed, but
in general $C_x=\{f:\sup\{y:f(y)\leq 1\}=x\}\cup A_{\delta_0}$,
where $A_{\delta_0}=\{f:f(x)\in\{1,\infty\}\text{ for all }x\}$,
the set containing the MGF of the Dirac measure $\delta_0$, and all
functions mimicking it. This characterisation of $C_x$ is useful
as it helps us to prove certain properties of $\ILo$.

\subsection{Properties of $\ILo$}
It should be noted that a weak law can be proven for $\{\estLo\}$
without applying Theorem \ref{thm:loynes}, by showing that $
P(\estLo\in\cdot)$ is eventually concentrated on any set of the
form $[0,a)$ for any $a>\loynes_\mu$ or $(b,\infty]$ for
$b<\loynes_\mu$, and deducing concentration on their intersection.
Interestingly this is a necessary deduction, as the rate function
$\ILo$ may not have a unique zero. Consider the Exp$(\lambda_1,\lambda_2)$
distribution with $\lambda_1,\lambda_2>0$, defined by the measure
\begin{align*}
\mu(dx)=\begin{cases}
	\displaystyle \frac{\lambda_1}{2}e^{\lambda_1 x}dx &x<0 \\
	\displaystyle \frac{\lambda_2}{2}e^{-\lambda_2 x}dx &x\geq 0.
	\end{cases}
\end{align*}
This distribution has mean $1/2(1/\lambda_2-1/\lambda_1)$, and so
$\loynes_\mu>0$ if $\lambda_2>\lambda_1$. Also $\H(\nu|\mu)<\infty$
for any $\nu\sim$Exp$(\gamma_1,\gamma_2)$, and $f_\nu$ is finite
on $(-\gamma_1,\gamma_2)$. So by Theorem \ref{thm:MLDP}(c) any
$f$ that mimics $f_\mu$ satisfies $\IM(f)=0$, and so $\ILo(x)=0$
for all $x<\theta_\mu$. Note however, that we cannot have $\ILo(x)=0$
for $x>\theta_\mu$.

\begin{lemma}[Positive on $(\loynes_\mu,\infty{]}$]
\label{lemma:positive}
$\ILo(x)>0$ for all $x>\loynes_\mu$.
\end{lemma}
\begin{proof}
See Section \ref{sec:proofs}.
\end{proof}
A necessary and sufficient condition for a unique zero is stated below.
\begin{proposition}[Conditions for unique zero]
\label{prop:loynesuniquezero}
$\ILo(x)>0$ for $x<\theta_\mu$ if and only if the random walk
associated with $e^{yX_1}$ satisfies the conditions of Crame\'r's
Theorem for some $y>x$. Therefore $\ILo$ has a unique zero if and
only if the random walk associated with $e^{yX_1}$ obeys Cram\'er's
Theorem for all $y\in (0,\loynes_\mu)$.
\end{proposition}
\begin{proof}
See Section \ref{sec:proofs}.
\end{proof}
Although the condition in Proposition \ref{prop:loynesuniquezero} may seem strict, it is true for
any distribution bounded above. \\ \\
Without too much analysis, we can prove a number of properties of $\ILo$. 
\begin{theorem}[Properties of $\ILo$]
\label{thm:propertiesofILo}
$\:$\\
\begin{enumerate}[(a)]
\item $\ILo$ is increasing on $[\theta_\mu,\infty]$ and decreasing on $[0,\theta_\mu]$, 
\item $\ILo$ is finite everywhere and therefore bounded,
\item For all $x\in(\theta_\mu,\infty)$, $\ILo(x)=\IM(f)$ for some $f$ satisfying $f(x)=1$,
\item If $y$ is the smallest value for which $\ILo(y)=0$, and moreover $y>0$ then $\ILo$ is continuous at all $x\neq y$. If $y=0$ then $\ILo$ is continuous everywhere.
\end{enumerate}
\end{theorem}
\begin{proof}
See Section \ref{sec:proofs}.
\end{proof}
After considering the proof of Theorem \ref{thm:propertiesofILo}(d) it should be clear that the methods involved will not suffice to prove continuity at the smallest value of $x$ for which $\ILo(x)=0$ if $x>0$, as it is possible that $\ILo(x)=\IM(f)$ for some $f$ satisfying $f(x)<1$. In fact this is true if $x<\theta_\mu$. However it is easy to prove continuity for this $x$ if $f(x)=1$ or if there exists some $g\in\mathcal{D}_{[-\infty,x]}$ with $g(x)=\infty$ and $\IM(g)<\infty$, as we can then use $g$ instead of $\mup$ in the proof that $\lim_{\epsilon\downarrow 0}\ILo(x-\epsilon)\leq\ILo(x)$, because although $f(x-\epsilon)\not\rightarrow 1$, $g(x-\epsilon)\rightarrow\infty$ and so $a_\epsilon\rightarrow 1$, as required. Whether or not there exists a $\mu$ such that $\ILo$ is discontinuous at this point is not considered further here. \\

\section{Proofs}
\label{sec:proofs}
\subsection{Section 4}

\begin{proof} [{\bf Proposition \ref{prop:limitsprop}. Characterization of possible limits of MGFs in $\MSpacebig$.}]
Consider any $f\in\MSpacebig$ that is the limit of a sequence of 
MGFs $\{f_n\}\subset\BSpace$, with corresponding
distribution functions $\{F_n\}$. Since $f_n\in\BSpace$, each $F_n$
is uniquely determined. An application of the Helly Selection Principle
tells us that if we have a sequence of distribution functions
$\{F_n\}$, then a subsequence of them converge point-wise to a
function $F$. We can replace $\{F_n\}$
with this convergent subsequence and so replacing $\{f_n\}$ with
the corresponding subsequence of MGFs, we
find they still converge to $f$. $F$ can be seen to be monotonic increasing with
codomain $[0,1]$, and therefore can only have countably many
discontinuities. Its upper semi-continuous regularisation \cite{Lewis95}, denoted
$F^\diamond$, is therefore a distribution function when the sequence
of probability measures corresponding to $F_n$ is tight, i.e. when
$\sup_{x\in\R}\inf_{n}\left(F_n(x)-F_n(-x)\right)=1$.
Since $F$ and $F^\diamond$ only have a countable number
of discontinuities they are equal almost everywhere with respect
to Lebesgue measure and so this,
along with monotonicity, can be used to show that they have the
same limit as $x\rightarrow\pm\infty$. 

Here we
deal with three cases regarding the limit of the distribution
functions, which are exhaustive and completely characterise the
possible limit functions $F$ and $f$. Let $a=\lim_{x\rightarrow-\infty}F(x)$ and $b=\lim_{x\rightarrow\infty}F(x)$. It suffices to prove the following. If 
\begin{align*}
a>0 \text{ and } b<1,
\end{align*}
then $f$ is of Type 1 or 2.
If 
\begin{align*}
a=0\text{ and } b<1,
\text{ or }
a>0\text{ and } b=1,
\end{align*}
then $f$ is of Type 2. 
If 
\begin{align*}
a=0\text{ and } b=1,
\end{align*}
then $f$ is of Type 1 or 3. The proof of this final statement generalises
ideas presented in \cite{Kozakiewicz47}, which does not deal with
the case of moment generating functions that were not finite in a
neighbourhood of the origin.

In the first case, it holds that 
\begin{align*}
&\text{for all }\epsilon>0,\:x\in \R\:\text{ there exists } N_{\epsilon,\:x}\text{ such that }F_n(x)<F(x)+\epsilon\leq b+\epsilon\:\text{ for all }\:n>N_{\epsilon,\:x}.
\end{align*}
Fix $\epsilon>0$, $x\in \R$, $\theta>0$, and see that for $n>N_{\epsilon,\:x}$,
\begin{align*}
f_n(\theta)&=\int_{-\infty}^{x} e^{\theta y}dF_n(y)+\int_{x}^\infty e^{\theta y}dF_n(y)\geq \int_{x}^\infty e^{\theta y}dF_n(y)\geq e^{\theta x}(1-F_n(x))> e^{\theta x}(1-(b+\epsilon)).
\end{align*} 
This is true for any $x$, so in particular we can increase $x$
without bound and this is still true, although it does change
$N_{\epsilon,\:x}$. This gives us
$\lim_{n\rightarrow\infty}f_n(\theta)=\infty$ as we can choose
$\epsilon$ so that $1-(b+\epsilon)>0$. An analogous argument using
$a>0$ will give us the same result for $\theta<0$. Thus the $\taw$
limit of $f_n$ must also satisfy $f(\theta)=\infty$ for $\theta\neq
0$. It can be shown easily that $f(0)\leq 1$; assume that
$f(0)>1$. Then $(0,1)\not\in\epi(f)$, so that $d((0,1),\epi(f))=\delta>0$.
Then as $(0,1)\in\epi(f_n)$ for all $n$, we can show that for
$\epsilon<\delta$ and any set $B\ni(0,1)$, (\ref{eq:topology}) does
not hold for any $n$, so that we cannot have $f_n\rightarrow f$.
Therefore $f(0)\leq 1$; $f$ is of Type 2 if $f(0)<1$, and of Type
1 if $f(0)=1$.

For the second case, we can assume $b<1$ and $a=0$, as proving the other
case is analogous. In this case we still have $f(\theta)=\infty$
for $\theta>0$ using the arguments from the previous case.
Assume that $f(\theta)<\infty$ for some $\theta<0$, as otherwise
we are back to the aforementioned case of the function infinite
everywhere but at 0. We have that
\begin{align*}
&\text{ for all }\epsilon>0,\text{ there exists }x_\epsilon\text{ such that }F(x_\epsilon)>b-\epsilon, \\
&\text{ for all }\epsilon>0,\text{ there exists }N_{\epsilon}\text{ such that }F_n(x_\epsilon)>F(x_\epsilon)-\epsilon\text{ for all }n>N_{\epsilon}, \\
\Rightarrow &\text{ for all }\epsilon>0,\text{ there exists }x_\epsilon,N_{\epsilon}\text{ such that }F_n(x_\epsilon)>b-2\epsilon\text{ for all }n>N_\epsilon. 
\end{align*}
Similarly as before we have that
\begin{align*}
&\text{for all }\epsilon>0,\text{ there exists }N_{\epsilon}\text{ such that }F_n(x_\epsilon)<b+\epsilon\text{ for all }n>N_\epsilon.
\end{align*}
We can assume that $x_\epsilon\rightarrow \infty$ as $\epsilon\rightarrow 0$. Fixing $\epsilon>0$, and $\theta<0$ in the interior of $\mathcal{D}_f$, see that for $n>N_\epsilon$ where $N_\epsilon$ fits both of the above conditions,
\begin{align*}
f_n(\theta)&=\int_{-\infty}^{x_\epsilon}e^{\theta x}dF_n(x)+\int_{x_\epsilon}^{\infty}e^{\theta x}dF_n(x)\leq \int_{-\infty}^{x_\epsilon}e^{\theta x}dF_n(x)+e^{\theta x_\epsilon}(1-F_n(x_\epsilon)) \\
&<\int_{-\infty}^{x_\epsilon}e^{\theta x}dF_n(x)+e^{\theta x_\epsilon}(1-(b-2\epsilon)).
\end{align*}
For each $n$, define $X_n^*$ by the distribution function $F_n^*(x)=F_n(x)/F_n(x_\epsilon)$ for $x\leq x_\epsilon$. Then $dF_n(x)=F_n(x_\epsilon)dF_n^*(x)$ and
\begin{align*}
f_n(\theta)&< F_n(x_\epsilon) E(e^{\theta X_n^*})+e^{\theta x_\epsilon}(1-(b-2\epsilon)) \\
&< (b+\epsilon) E(e^{\theta X_n^*})+e^{\theta x_\epsilon}(1-(b-2\epsilon)).
\end{align*}
It can similarly be shown that 
\begin{align*}
f_n(\theta)
	\geq F_n(x_\epsilon) E(e^{\theta X_n^*})
	>(b-2\epsilon) E(e^{\theta X_n^*}).
\end{align*}
If we look at the point-wise limits of these functions, this gives us 
\begin{align*}
(b-2\epsilon)\limsup_{n\rightarrow\infty} E(e^{\theta X_n^*})\leq \limsup_{n\rightarrow\infty}f_n(\theta)&\leq (b+\epsilon)\limsup_{n\rightarrow\infty} E(e^{\theta X_n^*})+e^{\theta x_\epsilon}(1-(b-2\epsilon)) \\
\Rightarrow b\limsup_{\epsilon\rightarrow 0}\limsup_{n\rightarrow\infty} E(e^{\theta X_n^*})\leq \limsup_{n\rightarrow\infty}f_n(\theta)&\leq b\limsup_{\epsilon\rightarrow 0}\limsup_{n\rightarrow\infty} E(e^{\theta X_n^*}) \\
\Rightarrow \limsup_{n\rightarrow\infty}f_n(\theta)&=b\limsup_{\epsilon\rightarrow 0}\limsup_{n\rightarrow\infty} E(e^{\theta X_n^*}).
\end{align*}
The $\taw$ limit $f$ and $\{f_n\}$ are convex and continuous on the interior of $\effdomf$ and $\{\mathcal{D}_{f_n}\}$ respectively, so the point-wise limit of $f_n(\theta)$ exists and equals $f(\theta)$ on the interior of $\effdomf$, which can be shown by a minor modification of \cite{Beer93}[Lemma 7.1.2]. Therefore we have 
\begin{align*}
f(\theta)&=b\limsup_{\epsilon\rightarrow 0}\limsup_{n\rightarrow\infty} E(e^{\theta X_n^*}).
\end{align*}
As limit superiors of convex functions are convex,
$\limsup_{\epsilon\rightarrow 0}\limsup_{n\rightarrow\infty}
E(e^{\theta X_n^*})$ is a non-negative convex function for all
$\theta\in \R$ with value 1 at $\theta=0$ and finite on some interval
in $(-\infty,0)$. Therefore we know that $\lim_{\theta\uparrow
0}\limsup_{\epsilon\rightarrow 0}\limsup_{n\rightarrow\infty}
E(e^{\theta X_n^*})\leq 1$. So as $f$ is lower semi-continuous and
convex,
\begin{align*}
f(0)=\lim_{\theta\uparrow 0}f(\theta)=b\lim_{\theta\uparrow 0}\limsup_{\epsilon\rightarrow 0}\limsup_{n\rightarrow\infty} E(e^{\theta X_n^*})\leq b,
\end{align*}
so $f$ is of Type 2. 

In the third case the usc-regularisation of
$F$, $F^{\diamond}$ is a monotonic right-continuous function, also
with the above limits, and so is a distribution function. It is
from this point that we follow the work in $\cite{Kozakiewicz47}$.
As again we need not consider the case of $f$ finite only at 0, let $f$ be finite for some interval on the left of the origin. To be precise,
let $f$ be finite on $(\alpha,0]$, but not on $(-\infty,\alpha)$ (if $-\infty<\alpha$). We cannot apply \cite{Mukherjea06}[Theorem 2] as we cannot assume that $f$ is a moment generating function. We also cannot apply \cite{Kozakiewicz47}[Theorem 2(a)] because we cannot assume $f$ is finite in a neighbourhood of 0, but we do not need to. Instead,
we show that if $M(x)=\sup_n F_n(x)$, that
\begin{align}
\lim_{x\rightarrow-\infty}M(x)e^{tx}=0\text{ for all }t\in(\alpha,0). \label{Msmall}
\end{align}
Choose $t$ and $\gamma$ so that $\alpha<\gamma<t<0$, and see that for $x
<0$,
\begin{align*}
F_n(x)&=\int_{-\infty}^{x}dF_n(u) 
\leq \int_{-\infty}^{x}e^{\gamma (u-x)}dF_n(u) 
\leq e^{-x\gamma}f_n(\gamma), \\
\Rightarrow M(x)e^{tx}&\leq e^{x(t-\gamma)}\sup_{n}f_n(\gamma).
\end{align*}
As  $\gamma$ is in the interior of $\effdomf$, $\{f_n(\gamma)\}$ converges and is finite for all $n$, so the supremum over $n$ is finite, and taking the limit as $x\rightarrow-\infty$ reduces the above expression to 0, as $t-\gamma>0$. Now let $\theta$ and $\gamma$ be such that $\alpha<\gamma<\theta<0$, and let $M_{\gamma}=\sup_{x<0}M(x)e^{\gamma x}$. By the above we have shown that $M_{\gamma}$ is finite. Using integration by parts, and assuming $-N$ is a continuity point of $F_n$, we see that
\begin{align*}
\int_{-\infty}^{-N} e^{x\theta}dF_n(x)&= e^{x\theta}F_n(x)|_{-\infty}^{-N}-\theta\int_{-\infty}^{-N} e^{x\theta}F_n(x)dx \\
&\leq M(-N)e^{-N\theta}-0-\theta\int_{-\infty}^{-N} e^{x(\theta-\gamma)}e^{\gamma x}M(x)dx \\
&\leq M(-N)e^{-N\theta}-\theta \int_{-\infty}^{-N} e^{x(\theta-\gamma)}M_\gamma dx \\
&=M(-N)e^{-N\theta}-M_\gamma\frac{\theta}{\theta-\gamma}e^{-N(\theta-\gamma)}.
\end{align*}
By $(\ref{Msmall})$, we can make this expression as small as we
want for large enough $N$. Also, as $F$ also satisfies $F(x)\leq
M(x)$, we can do the same if we replace $dF_n$ with $dF$. Moreover,
as $F$ and $F^\diamond$ only disagree on at most countably many
points, we can choose a large $N$ (specifically so that $-N$ is a
point of continuity of $F$ and $F^\diamond$) so that integrating
with respect to $dF$ is equivalent to integrating with respect to
$dF^\diamond$. Finally, as $F_n\rightarrow F$ point-wise and therefore
to $F^\diamond$ weakly, $-N$ is a point of continuity of $F$ and
$F^\diamond$, and $e^{\theta x}$ is bounded on $x\in[-N,\infty)$
we have
\begin{align*}
\int_{-N}^\infty e^{\theta x}dF_n(x)\rightarrow\int_{-N}^\infty e^{\theta x}dF(x)=\int_{-N}^\infty e^{\theta x}dF^\diamond(x)
\end{align*}
point-wise for each $\theta\in(\gamma,0)$. Combining this with the
above result for the integral from $-\infty$ to $-N$ gives us that
$f_n(\theta)\rightarrow \fmimic(\theta)$ point-wise on $(\gamma,0)$,
where $\fmimic$ is the MGF corresponding to $F^\diamond$. But we
can make $\gamma$ as close to $\alpha$ as we want, so we have
point-wise convergence on $(\alpha,0)$. As discussed before $\taw$
convergence implies that $f_n\rightarrow f$ pointwise
on $(\alpha,0)$, so that
$f(\theta)=\fmimic(\theta)$ for $\theta\in(\alpha,0)$. Moreover as
$f$ and $\fmimic$ are convex and lower-semicontinuous,
$f(0)=\lim_{\theta\uparrow 0}f(\theta)=\lim_{\theta\uparrow
0}\fmimic(\theta)=1$ and $f(\alpha)=\lim_{\theta\downarrow
\alpha}f(\theta)=\lim_{\theta\downarrow
\alpha}\fmimic(\theta)=\fmimic(\alpha)$ if $\alpha$ is finite, even
if $\fmimic(\alpha)=\infty$. If $f$ was finite on $[0,\beta)$ but
not on $(\beta,\infty)$ we could similarly show that
$f(\theta)=f_\nu(\theta)$ for $\theta\in[0,\beta]$, and so in general
$f(\theta)=f_\nu(\theta)$ for all $\theta\in\effdomf$. If
$\mathcal{D}_f=\mathcal{D}_{f_\nu}$ then $f=f_\nu$ and $f$ is a
moment generating function, and so is of Type 1. Otherwise $f$
mimics $f_\nu$ and so is of Type 3.

\end{proof}

\begin{proof}[{\bf Proposition \ref{prop:Lessthan1}. Functions with infinite rate}]
If we define $\X_{INT}$ to be the space of all moment generating
functions finite on some open interval, equipped with the subspace
topology, and $\Measures$ to be the space of probability measures on $\R$ equipped with the weak topology, then the mapping from $\X_{INT}$ into the space of probability measures,
$\Phi:\X_{INT}\rightarrow\Measures$, is well
defined, as seen in \cite{Mukherjea06}[Theorem 2]. 
Let $G_m$ be a descending countable base containing $f$, and let
$\Phi(G_m\cap\BSpace)=A_m$. Let $\estL$ be the empirical measure of $\{X_i\}_{i=1}^n$, and see that
\begin{align*}
\inf_{G\ni f}\mover(G)&= \lim_{m\rightarrow\infty}\limsup_{n\rightarrow\infty}\frac{1}{n}\log P(\estM\in G_m) \\
&= \lim_{m\rightarrow\infty}\limsup_{n\rightarrow\infty}\frac{1}{n}\log P(\estM\in G_m\cap\BSpace) \\
&= \lim_{m\rightarrow\infty}\limsup_{n\rightarrow\infty}\frac{1}{n}\log P(\estL\in A_m) \\
&\leq -\lim_{m\rightarrow\infty}\inf_{\nu\in \overline{A_m}}\H(\nu|\mu) \\
&=-\inf_{\nu\in C}\H(\nu|\mu)
\end{align*}
where $C=\cap_{m=1}^\infty \overline{A_m}$. The last line follows
from \cite{Dembo98} [Lemma 4.1.6 (b)]. Let $\nu\in C$. As $\Measures$ is metrizable, let $U_m$ be a descending
countable base for $\nu$, and for each $m$ let
$\{\nu_{m,n}\}_{n=1}^\infty\subset A_m$ be a sequence satisfying
$\nu_{m,n}\rightarrow \nu$ as $n\rightarrow\infty$. If $N_{m}$ is
such that for $n\geq N_{m}$, $\nu_{m,n}\in U_m$, and if $\nu_m^*=\nu_{m,N_{m}}$,
$\nu_m^*\in A_m$ and $\nu_m^*\in U_m$, so that $\nu_m^*\rightarrow
\nu$. For each $\nu_m^*$ there is a corresponding $f_m^*\in G_m$,
and so $f_m^*\rightarrow f$. If $f(0)\neq 1$ and $f_m^*\rightarrow f$,
then by the proof of Proposition \ref{prop:limitsprop} a subsequence of their
corresponding distribution functions tends to a function that is
not a distribution function, and so their measures do not converge.
By contradiction, it must follow that $C$ is empty, and so $\inf_{G\ni
f}\mover(G)=-\infty$.

\end{proof}

\begin{proof}[{\bf Proposition \ref{prop:Abase}. Local convex base.}]

In order to establish Proposition \ref{prop:Abase}, it suffices
to demonstrate:
\begin{enumerate}
\item
\label{Abasei}
For each $f$ such that $f(0)=1$ and $\effdomf\neq\{0\}$ and each $k\in\N$, 
$f\in A_k(f)$.
\item
\label{Abaseii}
For each $f$ such that $f(0)=1$ and $\effdomf\neq\{0\}$ and each $k\in\N$, 
$A_k(f)$ is open.
\item
\label{Abaseiii}
For each $f$ such that $f(0)=1$ and $\effdomf\neq\{0\}$ and each $k\in\N$, 
$A_k(f)$ is convex.
\item
\label{Abaseiv}
For each $f$ such that $f(0)=1$ and $\effdomf\neq\{0\}$ and each $k\in\N$, 
$A_{2k}(f)\subset V_k(f)$.
\end{enumerate}

As our construction of $W(\fs_k)$ is based on nested
epigraphs, the following estimate, readily deducible from \cite{Beer93}[Lemma 1.5.1],
will prove useful. As $\epi(\fs)\subset\epi(f)$,
\begin{align}
\label{eq:beerest}
\sup_{x\in\Bc_k}|d(x,\epi(\fs))-d(x,\epi(f))|
&=\sup_{x\in\Bc_k}d(x,\epi(\fs)\cap \Bc_{2k+2})-d(x,\epi(f)\cap \Bc_{2k+2})
	\nonumber\\
&\leq e_d(\epi(f)\cap \Bc_{2k+2},\epi(\fs)\cap \Bc_{2k+2}) \nonumber\\
&= \sup_{x\in\epi(f)\cap \Bc_{2k+2}}d(x,\epi(\fs)\cap \Bc_{2k+2}),
\end{align}
where $e_d(A,B)$ is the excess between two sets $A$ and $B$. The first equality comes from the fact that $x\in \Bc_k$ and $(0,2)\in\epi(\fs)$, so that $d(x,\epi(\fs))=d(x,(a,b))\leq d(x,(0,2))$ for some $(a,b)\in\epi(\fs)$, and so $(a,b)\in \Bc_{2k+2}$. The same is true for $f$. Also, as $\epi(\fs)\subset \epi(f)$, $d(x,\epi(\fs))\geq d(x,\epi(f))$ for all $x\in\R^2$, which justifies the removal of the absolute value.

For item \ref{Abasei}, by construction, $\fl<\fsl<\fsr<\fr$ and
$f(\theta)<\fs_k(\theta)$ for all $\theta\in[\fsl,\fsr]$ so that
$f\lll \fs_k$ and thus $f\in W(\fs_k)$ for all $k$.
It remains to show that $f$ is also an element of $V_k(\fs_k)$.
From equation \eqref{eq:beerest} it suffices to prove that
\begin{align}
\label{eq:finA}
e_d(\epi(f)\cap \Bc_{2k+2}, \epi(\fs_k)\cap \Bc_{2k+2})<1/k,
\end{align}
which will follow from the construction of $\fs_k$. It suffices
to consider $x=(\theta,f(\theta))$ for $\theta\in[-2k-2,2k+2]$. If
$\theta\in[\fsl,\fsr]$, $d(x,\epi(\fs_k))\leq 1/(2k)$. If
$\theta\in[\fBl,\fsl]$ or $\theta\in[\fsr,\fBr]$, defined in equation
\eqref{eq:fB}, more care is needed. Consider the former and
note that by the triangle inequality
\begin{align*}
d((\theta,f(\theta)), \epi(\fs_k)) \leq 
	d\left((\theta,f(\theta)), \left(\fsl,f(\fsl)\right)\right)
	+ d\left(\left(\fsl,f(\fsl)\right), \epi(\fs_k)\right).
\end{align*}
By equation \eqref{eq:overhangs}, the first term is less than $1/(2k)$
and, as above, the second term is less than or equal to $1/(2k)$ so that
equation \eqref{eq:finA} is ensured and $f\in A_k(f)$.

For item \ref{Abaseii},
the set $V_k(\fs_k)$ is open as it is an element of a local base
for $\fs_k$, so it suffices to show that $W(\fs_k)$ is open. For
ease of notation, we shall show that $W(f)$ is open for any $f$
such that $-\infty<\fl<\fr<\infty$. To do this, let $g\in W(f)$.
We shall construct a set $V_m(g)$ such that $V_m(g)\subset W(f)$.

As $g\lll f$ and $g$ is convex, $g$ is continuous on $[\fl,\fr]$
so that $f-g$ is lower semi-continuous on this range and therefore
its infimum is attained and positive:
\begin{align*}
\delta = \inf_{\theta\in[\fl,\fr]}(f(\theta)-g(\theta))>0.
\end{align*}
Let
\begin{align*}
&0<\epsilon<\min\left(\fl-\gl,\gr-\fr, \frac{\delta}{2},1\right) 
\text{ be such that }\\
&\max\left(g(\fl-\epsilon)-g(\fl),g(\fr+\epsilon)-g(\fr)\right)
< \frac{\delta}{2}.
\end{align*}
By convexity, the second condition ensures that
$g(\theta\pm\epsilon)-g(\theta)<\delta/2$ for any $\theta\in[\fl,\fr]$.
Now choose $m\in\N$ so that
\begin{align*}
(\theta,g(\theta))\in B_m, \text{ for all } 
	\theta\in[\fl-\epsilon,\fr+\epsilon],\text{ and }\frac{1}{m}<\epsilon.
\end{align*}
Thus for any $h\in V_m(g)$,
\begin{align*}
\sup_{x\in B_m} |d(x,\epi(g))-d(x,\epi(h))| < \frac 1m  
\text{ and so }
\sup_{\theta\in [\fl-\epsilon,\fr+\epsilon]} d((\theta,g(\theta)),\epi(h)) 
	< \epsilon.
\end{align*}
So, for any $\theta\in [\fl-\epsilon,\fr+\epsilon]$ there exists
$(\xt,\yt)\in\epi(h)$ such that
\begin{align}
\label{eq:handybound}
\max(|\theta-\xt|,|g(\theta)-\yt|)<\epsilon.
\end{align}
In particular, consider $\theta=\fl-\epsilon$, then
$\xt<\fl-\epsilon+\epsilon=\fl$ and $h(\xt)\leq \yt<
g(\fl-\epsilon)+\epsilon<\infty$ and so $\theta_{h,l}<\fl$. It can
similarly be shown that $\theta_{h,r}>\fr$.

To show that $h\lll f$, it remains to be proven that for any
$\theta\in[\fl,\fr]$, that $h(\theta)<f(\theta)$. This will follow
from convexity of $g$ and $h$. Consider such a $\theta$, then by
equation \eqref{eq:handybound} at $\theta\pm\epsilon$, since $\xtme\leq\theta\leq\xtpe$ we have that
\begin{align*}
h(\theta)\leq \max(h(\xtme),h(\xtpe))
	< \max(g(\theta-\epsilon),g(\theta+\epsilon))+\epsilon
	< \left(g(\theta)+\frac{\delta}{2}\right)+\frac{\delta}{2} 
	\leq f(\theta).
\end{align*}
So if $h\in V_m(g)$, then $h\in W(f)$ and thus $W(f)$ is open
as required.

For item \ref{Abaseiii},
assume $g,h\in A_k(f)$, let $\alpha\in[0,1]$ and let $\la(\theta)=\alpha g(\theta) + (1-\alpha)
h(\theta)$ for all $\theta$. We wish to show that $\la\in A_k(f)$. $\la\in W(\fs_k)$ as $\la(\theta) \leq \max\{g(\theta),h(\theta)\}$
for all $\theta$ and therefore $\la\lll\fs_k$. 

We must show that $\la\in V_k(\fs_k)$. Note that for any $x$,
$d(x,\epi(\la))\geq \min\{d(x,\epi(g)),d(x,\epi(h))\}$ as $\epi(l_\alpha)\subset\epi(g)\cup\epi(h)$.
Thus as $\la,g,h\in W(\fs_k)$, we have that
\begin{align*}
|d(x,\epi(\fs_k))-d(x,\epi(\la))|  &=d(x,\epi(\fs_k))-d(x,\epi(\la)) \\
&\leq d(x,\epi(\fs_k))-\min\{d(x,\epi(g)),d(x,\epi(h))\} \\
&=\max\{|d(x,\epi(\fs_k))-d(x,\epi(h))|,|d(x,\epi(\fs_k))-d(x,\epi(g))|\}.
\end{align*}
Taking the supremum over all $x\in \Bc_k$, as $g,h\in
V_k(\fs_k)$ the right hand side is less $1/k$ and so $\la\in
V_k(\fs_k)$.

Finally, for item \ref{Abaseiv},
we have that $A_{2k}(f)=V_{2k}(\fs_{2k})\cap W(\fs_{2k})\subset
V_{2k}(\fs_{2k})$ so that it suffices to show that $V_{2k}(\fs_{2k})\subset
V_k(f)$. Let $g \in V_{2k}(\fs_{2k})$ and, using the triangle
inequality, consider
\begin{align*}
&\sup_{x\in \Bc_k} |d(x,\epi(f))-d(x,\epi(g))| \\
&\leq \sup_{x\in \Bc_{2k}} |d(x,\epi(f))-d(x,\epi(g))| \\
&\leq 
	\sup_{x\in \Bc_{2k}} |d(x,\epi(f))-d(x,\epi(\fs_{2k}))| 
	+\sup_{x\in \Bc_{2k}} |d(x,\epi(\fs_{2k}))-d(x,\epi(g))| \\
&< \frac{1}{2k} + \frac{1}{2k} = \frac1k,
\end{align*}
and so $g\in V_k(f)$ as required.
\end{proof}

\begin{proof} [{\bf Proposition \ref{prop:super}. Super-additivity.}]
For each $m>n$, define the partial estimates
\begin{align*}
M_{n+1,m}(\theta) = \frac{1}{m-n}\sum_{i=n+1}^m \exp(\theta X_i)
\end{align*}
and note that, for all $\theta\in\R$, $M_{1,n+m}$ is a convex
combination of $M_{1,n}$ and $M_{n+1,n+m}$:
\begin{align*}
M_{1,n+m}(\theta) = \frac{n}{n+m} M_{1,n}(\theta) 
	+ \frac{m}{n+m} M_{n+1,n+m}(\theta).
\end{align*}
Assume $f$ is such that $f(0)=1$ and $\effdomf\neq\{0\}$, and consider fixed
$A_k(f)$. By the convexity of $A_k(f)$ proved in Proposition
\ref{prop:Abase}, we have that if $M_{1,n}\in A_k(f)$ and $M_{n+1,n+m}\in
A_k(f)$, then $M_{1,n+m}\in A_k(f)$, so that, by independence and
identical distribution of increments,
\begin{align*}
P(M_{1,n+m}\in A_k(f)) 
	&\geq P(M_{1,n}\in A_k(f), M_{n+1,n+m}\in A_k(f))\\
	&= P(M_{1,n}\in A_k(f)) P(M_{n+1,n+m}\in A_k(f))\\
	&= P(M_{1,n}\in A_k(f)) P(M_{1,m}\in A_k(f)).
\end{align*}
Thus the sequence $\{\log P(\estM\in A_k(f))\}$ is super-additive
and we have existence of the limit
\begin{align*}
\lim_{n\to\infty} \frac{1}{n} \log P(\estM\in A_k(f))
	= \sup_{n}\left(\frac{1}{n} \log P(\estM\in A_k(f))\right),
\end{align*}
as required.

If $f(0)=1$ and $f(\theta)=+\infty$, then we cannot use Proposition
\ref{prop:Abase}, but for this specific function we can show
that $V_k(f)\cap\BSpace$ is convex in much the same way that we showed that $A_k(f)$
is in Proposition \ref{prop:Abase} so that the result
follows as above. See that if $g,h\in V_k(f)\cap\BSpace$ then $\la=\alpha g+(1-\alpha)h\in\BSpace$. Also, as they all take the value 1 at the origin,  we have that $\epi(f)\subset\epi(g),\epi(h),\epi(\la)$, so
\begin{align*}
|d(x,\epi(f))-d(x,\epi(\la))|  &=d(x,\epi(f))-d(x,\epi(\la)) \\
&\leq d(x,\epi(f))-\min\{d(x,\epi(g)),d(x,\epi(h))\} \\
&=\max\left(d(x,\epi(f))-d(x,\epi(h))|,|d(x,\epi(f))-d(x,\epi(g))|\right).
\end{align*} 
Again, by taking the supremum over all $x\in\Bc_k$ we can show that $\la\in V_k(f)$, so that $V_k(f)\cap\BSpace$ is convex. As $\mathbb{P}[\estM\in V_k(f)]=\mathbb{P}[\estM\in V_k(f)\in\BSpace]$ it follows as for $A_k(f)$ that 
\begin{align*}
\lim_{n\to\infty} \frac{1}{n} \log P(\estM\in V_k(f))
	= \sup_{n}\left(\frac{1}{n} \log P(\estM\in V_k(f))\right).
\end{align*}
\end{proof}

\begin{proof}[{\bf Proposition \ref{prop:Mcomp}. Compactness.}]
The collection of all closed sets in $\R^2$ equipped with $\taw$
is compact. This can be deduced as a result of \cite{Beer93}[Theorem
5.13], which proves that the space is compact when equipped with
the Fell topology, and \cite{Beer93}[Exercise 10(b), pg 144], which
shows the Attouch-Wets and Fell topologies coincide on $\R^2$ as the
closed and bounded subsets of $\R^2$ are compact. Thus to prove
that $\{f\in\MSpacebig:f(0)\leq 1\}$ is compact, it suffices to show that it is
closed in this larger space.

To establish this, consider a sequence of functions $\{f_n\}\subset\MSpacebig$ such that $f_n(0)\leq 1$ and $\lim_{n\rightarrow\infty}\epi(f_n)=A$ in the
$\taw$ topology. That $A=\epi(f)$ for some $f$ can be readily shown
by assuming that there exists an $(x,y)\in A$
and $z>0$ such that $(x,y+z)\not\in A$, and showing
it must follow that the same holds for some $\epi(f_n)$, which cannot be true. Thus, to show that $\{f\in\MSpacebig:f(0)\leq 1\}$ is closed it is sufficient to show that $f=\lim_{n\rightarrow\infty} f_n$ satisfies $f\in\MSpacebig$ and $f(0)\leq 1$. Hence we
will prove that: (i) $f$ is lower semi-continuous; (ii) $f$ is non-negative; \newline (iii) $(0,1)\in
\epi(f)$; and (iv) $f$ is convex.

(i) This follows as $(2^{ \R^2},\tau_{AW_d})$ is compact and
therefore complete, so that $\epi(f)$ is closed.

(ii) Let $(x,y)\in\Bc_k$. As 
\begin{align*}
|d((x,y),\epi(f_n))-d((x,y),\epi(f))|\leq\sup_{z\in \Bc_k} |d(z,\epi(f_n))-d(z,\epi(f))|,
\end{align*}
it follows that $d((x,y),\epi(f_n))\rightarrow d((x,y),\epi(f))$ as $n\rightarrow\infty$. This is true for all $(x,y)\in\R^2$ and is a well-known feature of the Attouch-Wets topology. If $y<0$, $d((x,y),\epi(f_n))$ is bounded below by $|y|$ by the non-negativity of $f_n$. Therefore $d((x,y),\epi(f))\geq |y|$ and so $(x,y)\not\in\epi(f)$.

(iii) $d((0,1),\epi(f))=\lim_{n\rightarrow\infty}d((0,1),\epi(f_n))=\lim_{n\rightarrow\infty}0=0$. As $\epi(f)$ is 
closed it follows that $(0,1)\in\epi(f)$.

(iv) Fix $(x_1,y_1),
(x_2,y_2)\in \epi(f)$, $\alpha\in[0,1]$, and let
$\alpha(x_1,y_1)+(1-\alpha)(x_2,y_2)=(x_3,y_3)$. As $\epi(f_n)$ is closed, there exist some $\delta_i,\epsilon_i$, $i=1,2$ such that 
\begin{align*} 
d((x_i,y_i),\epi(f_n))&=d((x_i,y_i),(x_i+\delta_i,y_i+\epsilon_i))
\end{align*}
and $(x_i+\delta_i,y_i+\epsilon_i)\in \epi(f_n)$. 
Thus there exist $\delta,\epsilon$ satisfying
$\delta=\alpha\delta_1+(1-\alpha)\delta_2$, \newline
$\epsilon=\alpha\epsilon_1+(1-\alpha)\epsilon_2$ such that 
\begin{align*}
\alpha(x_1+\delta_1,y_1+\epsilon_1)+(1-\alpha)(x_2+\delta_2,y_2+\epsilon_2) 
	=(x_3,y_3)+(\delta,\epsilon)\in \epi(f_n)
\end{align*}
and so 
\begin{align*}
d((x_3,y_3),\epi(f_n))&\leq\max\{|\delta|,|\epsilon|\} 
\leq\max\{|\delta_1|,|\delta_2|,|\epsilon_1|,|\epsilon_2|\} \\
&=\max\{d((x_1,y_1),\epi(f_n)),d((x_2,y_2),\epi(f_n))\}.
\end{align*}
 Therefore $d((x_3,y_3),\epi(f))=\lim_{n\rightarrow\infty}d((x_3,y_3),\epi(f_n))=0$, and so $(x_3,y_3)\in \epi(f)$.
\end{proof}

\begin{proof} [{\bf Theorem \ref{thm:MLDPnonchar}. LDP for MGF estimators.}]
First see that $f\in\MSpacesmall$ if and only if it is of Type 1 or 3 as given in the statement of Proposition \ref{prop:limitsprop}, so that $\MSpacesmall=\overline{\BSpace}\cap\{f:f(0)=1\}$.  Proposition \ref{prop:Mcomp} gives that the measures are exponentially tight, while Propositions \ref{prop:Lessthan1} and \ref{prop:super} give
coincidence of the upper and lower deviation functions, and therefore by
\cite{Dembo98}[Lemma 1.2.18], we obtain the result for the
MGF estimators $\{\estM\}$. We can then
reduce this to an LDP in $\MSpacesmall$ by \cite{Dembo98}[Lemma 4.1.5(b)] as the
effective domain of $\IM$ is a subset of $\MSpacesmall$, which is measurable.
\end{proof}

\begin{proof}[{\bf Lemma \ref{lem:Lcont}. Continuity of $\l$}]
This proof follows from \cite{Duffy05b}[Proposition 3], although
not immediately as the map $\l$ fails that proposition's criteria.
We will bypass this difficulty by considering a sequence of
restrictions of $\l$ that satisfy the propositions criteria and
such that the images of elements of $\MSpacesmall$ under such
restrictions are equal to their image under $\l$ when intersected
with an increasingly large neighbourhood of the origin. First,
consider the function $g:\R^2\mapsto\R\times(0,\infty):(\theta,\psi)\mapsto
(\theta,\exp(\psi))$. Then $\epi(\l(f)) = g^*(\epi(f))$, where
$g^*:\Pc(\R\times(0,\infty))\mapsto\Pc(\R^2)$, with $\Pc$ denoting
the power set, and $g^*(D)=\{(\theta,\psi):g(\theta,\psi)\in D\}$
is the pull back. Although $g$ is bijective, continuous and maps
bounded sets to bounded sets, its inverse fails to be uniformly
continuous on bounded sets, so that the assumptions of
\cite{Duffy05b}[Proposition 3] do not hold.

Now consider $g_n=g\vert_{\R\times[-n,\infty)}$, the function $g$ restricted to the set $\R\times[-n,\infty)$ with codomain $\R\times[\exp(-n),\infty)$, and the corresponding pullback $g^*_n:\Pc(\R\times[\exp(-n),\infty))\mapsto\Pc(\R\times[-n,\infty))$. Then $g_n$ is bijective, continuous and with an inverse uniformly continuous on bounded subsets. Therefore by \cite{Duffy05b}[Proposition 3] $g^*_n$ is continuous. Now consider any sequence of closed subsets of $\R\times(0,\infty)$, $\{A_m\}$, all containing the point $(0,1)$ and converging in $(\Pc(\R\times(0,\infty)),\taw)$ to $A$, which, by properties of $\taw$ convergence, must also contain the point $(0,1)$. Then by \cite{Beer93}[Exercise 7.4.1] $A_m\cap(\R\times[\exp(-n),\infty))\rightarrow A\cap(\R\times[\exp(-n),\infty))$ in $(\Pc(\R\times(0,\infty)),\taw)$ as $m\rightarrow\infty$. By $(\ref{eq:topology})$ we can easily show that this implies convergence in $(\Pc(\R\times[\exp(-n),\infty)),\taw)$. Therefore $g^*_n(A_m\cap(\R\times[\exp(-n),\infty)))\rightarrow g^*_n(A\cap(\R\times[\exp(-n),\infty)))$ as $m\rightarrow\infty$ for any $n$. Note that $g^*_n(A_m\cap (\R\times[\exp(-n),\infty)))=g^*(A_m)\cap (\R\times [-n,\infty))$ when considered as elements of $\Pc(\R^2)$. Therefore (\ref{eq:topology}) holds for $\{g^*(A_m)\cap (\R\times [-n,\infty))\}_{m=1}^\infty$ for any $\epsilon>0$ and any bounded set $B\subset\R\times[-n,\infty)$. 

Now fix any bounded $B\subset\R^2$ and any $\epsilon>0$. Then $B\subset\R\times[-n,\infty)$ for some $n$, and for any $C\subset\R^2$ containing the point $(0,0)$, such as $g^*(A_m)$ and $g^*(A)$, $d(x,C)=d(x,C\cap B_{2r})$ for any $x\in B$ where $r$ is such that $B\subset B_r$, and so for $n>2r$ and any $x\in B$,
\begin{align*}
d(x,g^*(A_m)\cap (\R\times [-n,\infty)))&=d(x,g^*(A_m)\cap (\R\times [-n,\infty))\cap B_{2r})=d(x,g^*(A_m)\cap B_{2r}) \\
&=d(x,g^*(A_m)).
\end{align*}
The same if true for $g^*(A)$ and so (\ref{eq:topology}) also holds for $\{g^*(A_m)\}$, any $\epsilon>0$ and any $B\subset\R^2$, so that $g^*(A_m)\rightarrow g^*(A)$ in $(\Pc(\R^2),\taw)$. This is true for every sequence $\{A_m\}$ containing the point $(0,1)$, and so $g^*$ is continuous when its domain is restricted to sets containing the point $(0,1)$, which implies that the functional $L:\MSpacesmall\rightarrow\LSpace$ is continuous. 
\end{proof}
\begin{proof}[{\bf Theorem \ref{thm:LLDP}. LDP for CGF estimators.}]
This result follows via an application of the contraction principle
from the LDP for $\{\estM\}$ in $\MSpacesmall$ proved in Theorem 
\ref{thm:MLDPnonchar} through the function
defined in Lemma \ref{lem:Lcont}.
\end{proof}

\begin{proof}[{\bf Lemma \ref{lem:JCont}. Continuity of $\j$}]
Although the notation $\effdomf$ has thus far been used for MGFs, here we will use it for elements of $\LSpace$ to denote their effective domain. Fix some $f\in\LSpace$ satisfying $\effdomf\neq\{0\}$, $\effdomf\subset[0,\infty)$, and let $f_n\rightarrow f$ in $\taw$. Then $f_n\rightarrow f$ pointwise on the interior of $\effdomf$ and $f_n\rightarrow\infty$ pointwise outside $\overline{\effdomf}$. Therefore $\j(f_n)\rightarrow \j(f)$ pointwise on the interior of $\effdomf$, $\j(f_n)\rightarrow \infty$ outside $\effdomf\cup[0,\infty)$ and $\j(f_n)\rightarrow -\infty$ on $(-\infty,0)$. This can be extended to uniform convergence on compact subsets of the interior of $\effdomf$ by the convexity of $\j(f)$ and $\j(f_n)$ on $(0,\infty)$, and so by a minor modification of \cite{Beer93}[Lemma 7.1.2] we have that $\j(f_n)\rightarrow \j(f)$ in $(\JSpace,\taw)$, so that $\j$ is continuous at $f$. Proof is identical if $\effdomf\neq\{0\}$, $\effdomf\subset(-\infty,0]$, as $\j(f)$ and $\j(f_n)$ are concave on $(-\infty,0)$.

Now assume that $f$ is finite in a neighbourhood of the origin, and let $f_n\rightarrow f$ in $\taw$. Let $f^-$ and $f^+$ mimic $f$ with $\mathcal{D}_{f^-}=\effdomf\cap(-\infty,0]$, $\mathcal{D}_{f^+}=\effdomf\cap[0,\infty)$, and define $f^-_n,f^+_n$ similarly. It follows from \cite{Beer93}[Exercise 7.4.1] that $f^-_n\rightarrow f^-$, $f^+_n\rightarrow f^+$ in $\taw$, and so $\j(f^-_n)\rightarrow \j(f^-)$, $\j(f^+_n)\rightarrow \j(f^+)$. Defining $A^+_n=\epi(\j(f^+_n))$, $A^-_n=\epi(\j(f^-_n))$, $C^+_n=A^+_n\cap\{(x,y):x\geq 0,\text{ or } x=0,y\geq \j_n(f)(0)\}$, and $A^+$, $A^-$ and $C^+$ similarly, as $\{(x,y):x\geq 0,\text{ or } x=0,y\geq \j_n(f)(0)\}$ is closed it follows again by \cite{Beer93}[Exercise 7.4.1] that $C^+_n\rightarrow C^+$ in $\taw$. As $\epi(\j(f))=A^-\cup C^+$ and similarly for $\epi(f_n)$, for any fixed closed and bounded set $B\subset\R^2$, $x\in B$ and $n\geq 1$, without loss of generality we can assume $d(x,A^-_n\cup C^+_n)=d(x,A^-_n)$ and that $|d(x,\epi(\j(f)))-d(x,\epi(\j(f_n)))|=d(x,\epi(\j(f)))-d(x,\epi(\j(f_n)))$ so that
\begin{align*}
|d(x,\epi(\j(f)))-d(x,\epi(\j(f_n)))|&=d(x,\epi(\j(f)))-d(x,\epi(\j(f_n))) \\
&=d(x,A^-\cup C^+)-d(x,A^-_n\cup C^+_n) \\
&=d(x,A^-\cup C^+)-d(x,A^-_n) \\
&\leq d(x,A^-)-d(x,A^-_n) \\
&\leq |d(x,\epi(\j(f^-)))-d(x,\epi(\j(f^-_n)))|.
\end{align*}
This is true for all $n$ and for all $x\in B$, so that 
\begin{align*}
&\sup_{x\in B}|d(x,\epi(\j(f)))-d(x,\epi(\j(f_n)))| \\
\leq &\max\{\sup_{x\in B}|d(x,\epi(\j(f^-)))-d(x,\epi(\j(f^-_n)))|,\sup_{x\in B}|d(x,\epi(\j(f^+)))-d(x,\epi(\j(f^+_n)))|\},
\end{align*}
and so as $\j(f^-_n)\rightarrow \j(f^-)$ and $\j(f^+_n)\rightarrow \j(f^+)$, we can apply (\ref{eq:topology}) to show that $\j(f_n)\rightarrow \j(f)$ in $\taw$.

Now assume that $f\in\LSpace$ satisfies $\effdomf=\{0\}$. Then for any sequence $f_n\rightarrow f$, $\j(f_n)(\theta)\rightarrow\infty$ for all $\theta>0$ and $\j(f_n)(\theta)\rightarrow-\infty$ for all $\theta<0$. To prove that $\j(f_n)\rightarrow \j(f)$ in $\taw$, we will apply \cite{Beer93}[Theorem 3.1.7] that states that for a collection $\{A_n\}$ and $A$, non-empty closed subsets of a metric space $X$, $A_n\rightarrow A$ in $\taw$ if and only if for every $k\geq 1$,
\begin{align*}
\sup_{x\in A\cap \Bc_k(x_0)}d(x,A_n)\rightarrow 0\text{ and }\sup_{x\in A_n\cap \Bc_k(x_0)}d(x,A)\rightarrow 0
\end{align*}
as $n\rightarrow\infty$, where $x_0$ is any fixed point in $X$, $\Bc_k(x_0)$ is a closed ball of radius $k$ and centre $x_0$, and by convention $\sup_{x\in\emptyset}d(x,C)=0$ for any non-empty set $C$. 

In our case, $A_n=\epi(\j(f_n))$, $A=\epi(\j(f))$ and $\Bc_k(x_0)=\Bc_k$, the closed ball in $\R^2$ of radius $k$ around the origin. First consider $\sup_{x\in \epi(\j(f))\cap \Bc_k}d(x,\epi(\j(f_n)))$ for some $k\geq 1$. As $\epi(\j(f))\cap \Bc_k$ is closed, for all $n$ there exists some $(x_n,y_n)\in\Bc_k$ with $x_n\leq 0$ such that $d((x_n,y_n),\epi(\j(f_n)))=\sup_{x\in \epi(\j(f))\cap \Bc_k}d(x,\epi(\j(f_n)))$. As $x_n\in[-k,0]$ for all $n$, a subsequence of $x_n$ converges, say to some $x\in[-k,0]$. First assume $x<0$ and fix some $\epsilon>0$. Then along this subsequence, when $n$ is large enough so that $|x_n-x|<\epsilon$ and $\j(f_n)(x)<-k$, then as $y_n\geq-k$, $(x,y_n)\in\epi(\j(f_n))$ and so $d((x_n,y_n),\epi(\j(f_n)))\leq d((x_n,y_n),(x,y_n))<\epsilon$. This is true for all $\epsilon>0$ and so along any subsequence of $\{f_n\}$ such that $\{x_n\}$ converges to some $x<0$, $\sup_{x\in \epi(\j(f))\cap \Bc_k}d(x,\epi(\j(f_n)))\rightarrow 0$. Now assume a subsequence of $\{x_n\}$ converges to 0. Then for any $\epsilon>0$ and $n$ large enough so that $x_n<-\epsilon$ and $\j(f_n)(-\epsilon/2)<-k$, we again have that $d((x_n,y_n),\epi(\j(f_n)))<\epsilon$. Therefore $\sup_{x\in \epi(\j(f))\cap \Bc_k}d(x,\epi(\j(f_n)))\rightarrow 0$ along any convergent subsequence of $\{x_n\}$. If $\sup_{x\in \epi(\j(f))\cap \Bc_k}d(x,\epi(\j(f_n)))\not\rightarrow 0$, that is $\sup_{x\in \epi(\j(f))\cap \Bc_k}d(x,\epi(\j(f_n)))>\epsilon$ infinitely often for any $\epsilon>0$, then along this subsequence we can find a sub-subsequence such that $\{x_n\}$ converges, which would yield a contradiction. Therefore $\sup_{x\in \epi(\j(f))\cap \Bc_k}d(x,\epi(\j(f_n)))\rightarrow 0$ along the entire sequence.

Now consider $\sup_{x\in \epi(\j(f_n))\cap \Bc_k}d(x,\epi(\j(f)))$ for some $k\geq 1$. Assume that $\epi(\j(f_n))\cap \Bc_k\cap([0,\infty)\times\R)\neq\emptyset$ for any $n$, as if this is the case then $\sup_{x\in \epi(\j(f_n))\cap \Bc_k}d(x,\epi(\j(f)))=0$. If this is not true infinitely often then we still need to show convergence along that subsequence, and if it is not true only finitely often then convergence holds immediately. With this assumption, notice that for all $n$ there exists some $(x_n,y_n)\in\Bc_k$ with $x_n\geq 0$ such that $\sup_{x\in \epi(\j(f_n))\cap \Bc_k}d(x,\epi(\j(f)))=d((x_n,y_n),\epi(\j(f))=x_n$. It is easy to see that $x_n=\sup\{\theta\in[0,k]:f_n(\theta)\leq k\}$, and by lower semi-continuity of $\j(f_n)$ on $(0,\infty)$, $\j(f_n)(x_n)\leq k$. Also, $\j(f_n)(x_n)\leq k\Rightarrow f_n(x_n)\leq kx_n\leq k^2$, so that $\sup_{x\in \epi(f_n)\cap \Bc_{k^2}}d(x,\epi(f))\geq x_n$. But as $f_n\rightarrow f$, $\sup_{x\in \epi(f_n)\cap \Bc_k^2}d(x,\epi(f))\rightarrow 0$ so that $\sup_{x\in \epi(\j(f_n))\cap \Bc_k}d(x,\epi(\j(f)))=x_n\rightarrow 0$. 
\end{proof}

\begin{proof}[{\bf Theorem \ref{thm:JLDP}. LDP for Jarzynski estimators}]
This result follows from the continuity of $\j$ and the contraction principle. 
\end{proof}

\begin{proof}[{\bf Theorem \ref{thm:ILDP}. LDP for rate function estimators.}]
This result follows from the continuity of $\lf$ and the contraction principle. 
\end{proof}
\subsection*{Section 5}

\begin{proof}[{\bf Proposition \ref{lemma:convex}. Convexity of $\IM$.}]
We establish global convexity of $\IM$ by creating an argument along
the lines of \cite{Dembo98}[Lemma 4.1.21], but the conditions of
that Lemma as stated do not hold here as $\MSpacesmall$ is not a topological
vector space. The proof of \cite{Dembo98}[Lemma 4.1.21], however,
hinges only on two properties that we instead establish directly.

First, we need to show continuity of averaging in $\MSpacesmall$.  That
is, if $\{f_n\},\{g_n\}\subset \MSpacesmall$ are such that $f_n\rightarrow
f\in\MSpacesmall$ and $g_n\rightarrow g\in\MSpacesmall$ in $\taw$, then
$(f_n+g_n)/2\rightarrow (f+g)/2$ in $\taw$.
If $f$ or $g$ is finite and continuous at one point in the domain of the other, then we can apply \cite{Beer93}[Theorem 7.4.5] to show that $\lim_{n\rightarrow\infty}(f_n+g_n)=f+g$ in $\MSpacebig$. To show continuity of multiplication by 1/2, see that $\{(f_n+g_n)/2\}$ is a sequence in $\{h:h(0)\leq 1\}$ which is compact by Proposition \ref{prop:Mcomp}, so it has a convergent subsequence. Replacing $\{(f_n+g_n)/2\}$ with this subsequence with limit $h$ and applying \cite{Beer93}[Theorem 7.4.5] it follows that 
\begin{align*}
f+g=\lim_{n\rightarrow\infty}[(f_n+g_n)/2+(f_n+g_n)/2]=2h
\end{align*}
so that $h=(f+g)/2$. This is true for every subsequence of $\{(f_n+g_n)/2\}$, which lies in a compact set, and so it must follow that $(f_n+g_n)/2\rightarrow h=(f+g)/2$ and so multiplication by 1/2 is a continuous operation in $\MSpacebig$. It then follows that averaging is continuous in $\MSpacebig$ and therefore in $\MSpacesmall$. If it not the case that $f$ or $g$ is finite and continuous at one point in the domain of the other, then it must follow that $f(\theta)=\infty$ for $\theta>0$ and $g(\theta)=\infty$ for $\theta<0$, or vice versa. In this case $(f+g)/2$ is finite only at $0$ and $(f_n(\theta)+g_n(\theta))/2\rightarrow\infty$ point-wise for each $\theta\neq 0$, 
so $h=\lim_{n\rightarrow\infty}(f_n+g_n)/2$ satisfies $h(\theta)=\infty$ for $\theta\neq 0$, and so $h=(f+g)/2$ if $h(0)=1$. Consider $(0,y)$ for $y<1$. As $f(0)=1$, $d((0,y),\epi(f_n))\rightarrow d((0,y),\epi(f))>0$, and similarly for $d((0,y),\epi(g_n))$. Furthermore, 
\begin{align*}
(f_n(\theta)+g_n(\theta))/2&\geq\min\{f_n(\theta),g_n(\theta)\} \\
\Rightarrow \epi((f_n+g_n)/2)&\subset\epi(f_n)\cup\epi( g_n) \\
\Rightarrow d((0,y),\epi((f_n+g_n)/2))&\geq\min\{d((0,y),\epi(f_n),d((0,y),\epi(g_n))\} \\
\Rightarrow d((0,y),\epi((f_n+g_n)/2))&\not\rightarrow 0 \\
\Rightarrow d((0,y),\epi(h))&>0.
\end{align*}
Therefore $h(0)>y$ for any $y<1$, so $h(0)=1$ as required.  

Second, we need to show the continuity of $\beta f+(1-\beta)g$ with
respect to $\beta\in (0,1)$. First notice that $\mathcal{D}_{\beta
f+(1-\beta)g}=\mathcal{D}_f\cap\mathcal{D}_g$ for all $\beta\in(0,1)$.
Moreover we have pointwise convergence of $\beta f+(1-\beta)g$ on
the interior of $\mathcal{D}_f\cap\mathcal{D}_g$ as $\beta\rightarrow
\beta_0\in(0,1)$. This can be extended to uniform convergence on
bounded subsets of the interior of $\mathcal{D}_f\cap\mathcal{D}_g$
by the convexity of $\beta f+(1-\beta)g$. Convergence in $\taw$ follows
from \cite{Beer93}[Lemma 7.1.2], with simple modifications to handle
the possibility of open domains and common domains not necessarily equal to $\R$. Now see that for any sets $A_1$ and
$A_2$,
letting $(A_1+A_2)/2=\{g:g=(f_1+f_2)/2,\:f_1\in A_1,\:f_2\in A_2\}$
and following the notation of Proposition \ref{prop:super},
\begin{align*}
 P(\estM\in A_1,\:M_{n+1,2n}\in A_2)&\leq  P\left[M_{2n}\in \frac{A_1+A_2}{2}\right] \\
\Rightarrow  P(\estM\in A_1) P(M_{n}\in A_2)&\leq  P\left(M_{2n}\in \frac{A_1+A_2}{2}\right) \\
\Rightarrow \frac{1}{2n}\log P(\estM\in A_1)+\frac{1}{2n}\log P(M_{n}\in A_2)&\leq \frac{1}{2n}\log P\left(M_{2n}\in \frac{A_1+A_2}{2}\right) \\
\Rightarrow \frac{1}{2}\liminf_{n\rightarrow\infty}\frac{1}{n}\log P(\estM\in A_1)+\frac{1}{2}\liminf_{n\rightarrow\infty}\frac{1}{n}\log P(\estM\in A_2)&\leq \limsup_{n\rightarrow\infty}\frac{1}{n}\log P\left(M_{n}\in \frac{A_1+A_2}{2}\right).
\end{align*}
As $(f_1,f_2)\mapsto (f_1+f_2)/2$ is a continuous operation, as is $\beta\mapsto \beta f+(1-\beta)g$ for any $f,g\in \MSpacesmall$ and $\beta\in(0,1)$, the remainder of the proof follows identically to that of \cite{Dembo98}[Lemma 4.1.21]. 
\end{proof}

\begin{proof}[{\bf Proposition \ref{thm:I=H}. $\IM(f)$ for $f$ an MGF}]

First it is worth noting that $\nu$ is uniquely defined if $f$ is
finite on an open interval, which need not include $0$ \cite{Mukherjea06}.
As in Proposition \ref{prop:Lessthan1}, if we define $\X_{INT}$ to be the space of all moment generating
functions finite on some open interval, equipped with the subspace
topology, then the mapping from MGFs to measures
$\Phi:\X_{INT}\rightarrow \Measures$ is well
defined and continuous when $\X_{INT}$ is equipped with the
subspace topology, as convergence of MGFs that are finite on some
open interval to another MGF that is finite on some open
interval implies convergence of their distributions, as seen in \cite{Mukherjea06} and in the proof of Proposition \ref{prop:limitsprop}.
So we have that for every open $G\subset\Measures$, $\Phi^{-1}(G)=G'\cap\X_{INT}$ for some open
$G'\subset \MSpacesmall$. As $\Phi^{-1}(\nu)=f_\nu$ for $f_\nu\in\X_{INT}$, $G\ni\nu\Rightarrow G'\ni f_\nu$ and so if $\estL$ is the empirical distribution governed by the same sample as for $M_n$,
\begin{align*}
\liminf_{n\rightarrow\infty}\frac{1}{n}\log P(\estL\in G)
&=\liminf_{n\rightarrow\infty}\frac{1}{n}\log P(\estM\in G'\cap\X_{INT}) \\
&=\liminf_{n\rightarrow\infty}\frac{1}{n}\log P(\estM\in G') 
=\munder(G') 
\geq-\IM(f_\nu) \\
\end{align*}
This implies that
\begin{align*}
\inf_{G\ni\nu}\liminf_{n\rightarrow\infty}\frac{1}{n}\log P(\estL\in G)&\geq -\IM(f_\nu) 
\Rightarrow -\H(\nu|\mu)\geq -\IM(f_\nu).
\end{align*}

To prove the second inequality, that for any $f_\nu\in\BSpace$,
$\IM(f_\nu)\leq \H(\nu|\mu)$, let $\Y_N=\{\nu\in\Measures:\nu\text{
is supported on }[-N,N]\}$, and let $\Gamma$ be the mapping from
measures to MGFs. Since convergence of measures in $\Y_N$ implies
convergence of the corresponding moment generating functions, we
have that $\Gamma\vert_{\Y_N}$ is continuous and so
$G_N\cap\Y_N=\Gamma^{-1}(G'\cap\Gamma(\Y_N))$ is open for every
open set $G'\subset\MSpacesmall$. Note that we do not know if the
same $G\subset\Measures$ forms the inverse image of $G'\cap\Gamma(\Y_N)$
for every $N$, hence the use of the subscript, yet we can demand
that $G_N\subset G_{N+1}$. Note that
\begin{align*}
\Gamma^{-1}(G'\cap\Gamma(\Y_N))&=\Gamma^{-1}((G'\cap\Gamma(\Y_N))\cap(G'\cap\Gamma(\Y_{N+1}))) \\
&=\Gamma^{-1}(G'\cap\Gamma(\Y_N))\cap\Gamma^{-1}(G'\cap\Gamma(\Y_{N+1})) \\
&=(G_N\cap\Y_N)\cap(G_{N+1}\cap\Y_{N+1}) \\
&=(G_N\cap G_{N+1})\cap\Y_N.
\end{align*}
So if $G_N\not\subset G_{N+1}$, we can replace $G_N$ with $G_{N}\cap
G_{N+1}$.  So fixing $f_\nu\in\BSpace$, $f_\nu\in\Gamma(\Y_M)$ for
some $M$, and so we have that for every $G'\ni f_\nu$, $\nu\in G_M$
and so
\begin{align*}
P(\estM\in G')&= P(\estM\in G'\cap\BSpace) 
= P(\estM\in \cup_{N=M}^\infty (G'\cap \Gamma(\Y_N))) 
= P(\estL\in \cup_{N=M}^\infty(G_N\cap \Y_N)) \\
&\geq  P(\estL\in\cup_{N=M}^\infty( G_M\cap \Y_N)) 
=P(\estL\in G_M\cap\Gamma^{-1}(\BSpace)) 
= P(\estL\in G_M).
\end{align*}
This implies that
\begin{align*}
\munder(G')&\geq\liminf_{n\Rightarrow\infty}\frac{1}{n}\log P(\estL\in G_M) 
\geq \inf_{G\ni\nu}\liminf_{n\rightarrow\infty}\frac{1}{n}\log P(\estL\in G) 
=-\H(\nu|\mu).
\end{align*}
This is true for every $G'\ni f_\nu$, and so 
\begin{align*}
\inf_{G'\ni f_\nu}\munder(G)&\geq-\H(\nu|\mu) 
\Rightarrow -\IM(f_\nu)\geq -\H(\nu|\mu). 
\end{align*}

To prove the third inequality, that for any moment generating function
$f$, $\IM(f)\leq \inf_{\left\{\nu:f=f_\nu\right\}}\H(\nu|\mu)$,
fix any moment generating function $f_\nu\in\MSpacesmall$. Let
$\nu_n(dx)=\nu(dx)/\nu([n,n])$ on $[-n,n]$. It can easily be shown
that $\H(\nu_n|\mu)\rightarrow \H(\nu|\mu)$ as $n\rightarrow\infty$.
As discussed before, $f_{\nu_n}\rightarrow f_\nu$ in $\taw$,
$\IM(f_{\nu_n})=\H(\nu_n|\mu)$ by the previous two lemmas and so
by lower semi-continuity of $\IM$,
\begin{align*}
\IM(f_\nu)\leq \lim_{n\rightarrow\infty}\IM(f_{\nu_n})=\lim_{n\rightarrow\infty}\H(\nu_n|\mu)=\H(\nu|\mu).
\end{align*}
In the case of the moment generating function $f$ finite only at
0, this statement is true for all $\nu$ such that $f=f_\nu$, and
so $\IM(f)\leq\inf_{\nu:f=f_\nu}\H(\nu|\mu)$.
\end{proof}

\begin{proof}[{\bf Lemma \ref{lemma:mimics}. Characterizations for $f$ not a MGF}]

Point \ref{lemma:valuesforspike}.
Assume $\IM(f)<\infty$. Then for any $g\in\MSpacesmall$ and any any $\gamma\in(0,1)$, $\gamma
f+(1-\gamma)g=f$. By the convexity of $\IM$ proved in Proposition
\ref{lemma:convex},
$\IM(f)=\IM(\gamma f+(1-\gamma)g)\leq \gamma \IM(f)+(1-\gamma)\IM(g)$,
so that $\IM(f)\leq \IM(g)$.
This is true for all $g\in\MSpacesmall$, and so $\IM(f)=0$.

Point \ref{lemma:finitemimic}.
See first that $\gamma g+(1-\gamma) \fmimic\rightarrow f $ in $\taw$
as $\gamma\rightarrow 0$, as they converge uniformly on closed subsets of $(\alpha,\beta)$
and are infinite outside $[\alpha,\beta]$, and so $\taw$ convergence can be proven by a minor modification of \cite{Beer93}[Proposition 7.1.3]. Thus
\begin{align*}
\IM(f) \leq \lim_{\gamma\rightarrow 0}\IM(\gamma g+(1-\gamma)\fmimic)
	\leq \lim_{\gamma\rightarrow 0}\left(\gamma\IM(g)+(1-\gamma)\IM(\fmimic)\right)=\IM(f_\nu).
\end{align*}

Point \ref{lemma:valuesformimics}
First we show that $\IM(f)\geq \IM(\fmimic)$, in a manner very
similar to how we showed that $\IM(g)=\infty$ for $g(0)<1$ in Proposition \ref{prop:Lessthan1}, except
in this case it happens that $C$ is non-empty. Like before, let
$G_m$ be a descending countable open base of $f$, and let $A_m$ be
defined by $\Phi(G_m\cap\BSpace)=A_m$. See that
\begin{align*}
\IM(f)
&=-\lim_{m\rightarrow\infty}\limsup_{n\rightarrow\infty}\frac{1}{n}\log P(\estM\in G_m) 
= -\lim_{m\rightarrow\infty}\limsup_{n\rightarrow\infty}\frac{1}{n}\log P(\estL\in A_m) \\
&\geq \lim_{m\rightarrow\infty}\inf_{\kappa\in \overline{A_m}}\H(\kappa|\mu) 
=\inf_{\kappa\in C}\H(\kappa|\mu),
\end{align*}
where $C=\cap_{m=1}^\infty \overline{A_m}$. It can be shown that
$C$ is non-empty, however here it is unnecessary to do so, as if
$C$ is empty then $\IM(f)=\infty$ and so $\IM(f)\geq \IM(\fmimic)$
trivially. Let $\kappa\in C$. As $\Measures$ is metrizable, let $U_m$
be a descending countable base for $\kappa$, and let
$\{\kappa_{m,n}\}_{n=1}^\infty\subset A_m$ be a sequence satisfying
$\kappa_{m,n}\rightarrow \kappa$ as $n\rightarrow\infty$. Let $N_{m}$ be
such that for $n\geq N_{m}$, $\kappa_{m,n}\in U_m$. Then if
$\kappa_m^*=\kappa_{m,N_{m}}$, $\kappa_m^*\in A_m$ and $\kappa_m^*\in U_m$, so
that $\kappa_m^*\rightarrow \kappa$. For each $\kappa_m^*$ there is a
corresponding $f_m^*\in G_m$, and so $f_m^*\rightarrow f$. If $f$
mimics $\fmimic$, then we have a sequence of moment generating
functions converging to a function mimicking a moment generating
function $f_\nu$, and a corresponding convergent
sequence of measures, then it follows that $\kappa=\nu$ from the
proof of Proposition \ref{prop:limitsprop}, 
and from \cite{Mukherjea06}[Theorem 2].
So $C=\{\nu\}$ and thus $\IM(f)\geq \H(\nu|\mu)=\IM(\fmimic)$. Now assume that
$\IM(f)<\infty$, then we can apply statement \ref{lemma:finitemimic} of Lemma \ref{lemma:mimics}
with $f=g$ and so $\IM(f)=\IM(\fmimic)$.
\end{proof}

\begin{proof}[{\bf Lemma \ref{thm:5equiv}. Five equivalences.}]

$(1)\Rightarrow (2)$.
Assume (1), so that the sequence $1/n\sum_{1=1}^ne^{\gamma X_i}$
satisfies the conditions of Cram\'er's Theorem. Assume without loss of generality that
$\gamma<\alpha$, and so $\alpha\leq 0$ is finite. Fix $f\in\effdomsub$.
Then for $k$ large enough so that $\alpha-1/k>-k$ and $\gamma<\alpha-1/k$,
\begin{align*}
\estM\in V_k(f)\Rightarrow \estM(\alpha-1/k)>k\Rightarrow \estM(\gamma)>k
\end{align*}
as $d((\alpha-1/k,k),\epi(f))\geq 1/k$, so we can't have
$(\alpha-1/k,k)\in\epi(\estM)$. So
\begin{align*}
\mover(V_k(f))\leq \limsup_{n\rightarrow\infty}\frac{1}{n}\log P(\estM(\gamma)>k)\leq-\inf_{x\geq k}I_\gamma^*(x)
\end{align*}
where $I_\gamma^*(x)$, the Cram\'er's Theorem rate function
for $1/n\sum_{1=1}^ne^{\gamma X_i}$, has compact level sets. Then
\begin{align*}
-\IM(f)&=\lim_{k\rightarrow\infty}\mover(V_k(f))\leq \lim_{k\rightarrow\infty}-\inf_{x\geq k}I_\gamma^*(x)=-\infty \\
\Rightarrow \IM(f)&=\infty.
\end{align*}
This is true for all $f\in\effdomsub$, so $(1)\Rightarrow (2)$.

$(2)\Rightarrow (3)$ follows immediately from Proposition \ref{thm:I=H}.

$(3)\Rightarrow (4)$ is trivial as $\effdom\subset\effdomsub$. 

$(4)\Rightarrow (1)$ requires the lengthiest of proof. 
We first establish the result when $\mu$ is a discrete distribution.
Consider all $\nu$ such that $f_\nu\in \effdom$. Then assuming (4),
\begin{align*}
\sum_{k\geq 0}\nu(k)\log\left(\frac{\nu(k)}{\mu(k)}\right)=\infty\text{ for all such }\nu 
\text{ or }\sum_{k<0}\nu(k)\log\left(\frac{\nu(k)}{\mu(k)}\right)=\infty\text{ for all such }\nu,
\end{align*}
as otherwise, if $\nu$ and $\kappa$ were distributions with
$f_\nu,f_\kappa\in\effdom$ that give a finite sum in the first and
second case respectively, then a distribution $\pi$ satisfying
$\pi(k)\propto\nu(k)$ for $k\geq 0$ and $\pi(k)\propto\kappa(k)$
for $k<0$ would still have moment generating function in $\effdom$
and would have $\H(\pi|\mu)<\infty$. Without loss of generality
assume that
\begin{align}
\sum_{k\geq 0}\nu(k)\log\left(\frac{\nu(k)}{\mu(k)}\right)=\infty\text{ for all such }\nu.\label{eq:badtail}
\end{align}
If $\beta=\infty$, then for any $\nu$ with $f_\nu\in\effdom$,
$\kappa$ satisfying $\kappa(k)=\nu(k)/\nu(-\infty,0]$ for $k\leq
0$ and $\kappa(k)=0$ otherwise, also has $f_\kappa\in\effdom$, but
(\ref{eq:badtail}) does not hold. So it follows that $\beta<\infty$.
Also $f_\mu(\epsilon)<\infty$ for some $\epsilon>0$. To see this,
assume the contrary, that $f_\mu\in\mathcal{D}_{[\alpha^*,0]}$ for some $\alpha^*\leq 0$.
Then $\nu(k)\propto e^{-\beta k}\mu(k)$ for $k\geq 0$ has $f_\nu\in\effdom$
for some suitable behaviour of the left tail. But
\begin{align*}
\sum_{k\geq 0}\nu(k)\log\left(\frac{\nu(k)}{\mu(k)}\right)=\sum_{k\geq 0}e^{-\beta k}(-\beta k)\mu(k)<\infty
\end{align*}
as $e^{-\beta k}(-\beta k)\leq 0$ for all $k$. As this is impossible, it must follow that $f_\mu(\epsilon)<\infty$ for some $\epsilon>0$.

We can assume $X_1$
is not bounded above, as otherwise (1) holds trivially for any $\gamma>\beta$. Since the
support of $\mu$ has no upper bound we can find a $\nu$ supported
on a subset of the support of $\mu$ that satisfies $\lim_{k\rightarrow\infty}\nu(k)e^{\gamma k}=\infty$ for all $\gamma>0$, where the limit is taken along the support of $\nu$. To see this, let $\{b_n\}$ contained in the support $\mu$ satisfy $b_n>n$ for all $n$, and let $\{b_n\}$ be the support of $\nu$ with $\nu(b_n)\propto n^{-2}$. We can write $\nu$ as $\nu(k)=C_\nu
g(k)\mu(k)$ for some non-negative function $g$ defined on the support
of $\mu$. $C_\nu$ is a constant chosen so that $\nu$ is a distribution.
Without loss of generality we can set $g(k)\geq 3$ for all $k\geq
0$, as $\sum_{k\geq 0}(g(k)+3)\mu(k)<\infty$ if and only if
$\sum_{k\geq 0}g(k)\mu(k)<\infty$. As stated before $\sum_{k\geq
0}e^{\epsilon k}\mu(k)<\infty$ for small enough positive $\epsilon$ and
so
\begin{align*}
\lim_{k\rightarrow\infty}e^{\epsilon k}(g(k)+3)\mu(k)=\infty\Leftrightarrow\lim_{k\rightarrow\infty}e^{\epsilon k}g(k)\mu(k)=\infty,
\end{align*}
if the limit as $k\rightarrow\infty$ is taken along the support of
$\nu$. So $\nu$ still satisfies $\lim_{k\rightarrow\infty}\nu(k)e^{\gamma k}=\infty$ for all $\gamma>0$. Define $\kappa(k)=D_\nu\nu(k)/\log g(k)$
for all $k\geq 0$, where $D_\nu$ is chosen so that $\kappa$ is a
distribution. As $\log g(k)>1$, $\kappa(k)$ is summable and so $D_\nu$ exists. The behaviour of $\kappa(k)$ for $k<0$ is unimportant.
It follows that
\begin{align*}
\sum_{k\geq 0}\kappa(k)\log\left(\frac{\kappa(k)}{\mu(k)}\right)&=\sum_{k\geq 0}D_\nu\frac{\nu(k)}{\log g(k)}\log\left(D_\nu C_\nu\frac{g(k)}{\log g(k)}\right) \\
&\leq \sum_{k\geq 0}D_\nu\frac{\nu(k)}{\log g(k)}\log\left(D_\nu C_\nu\right)+\sum_{k\geq 0}D_\nu\frac{\nu(k)}{\log g(k)}\log\left(g(k)\right) \\
&\leq \sum_{k\geq 0}D_\nu\nu(k)\log\left(D_\nu C_\nu\right)+\sum_{k\geq 0}D_\nu\nu(k) \\
&<\infty.
\end{align*} 
So $\H(\kappa|\mu)<\infty$ and so $f_\kappa(\gamma)<\infty$ for some $\gamma\not\in[\alpha,\beta]$. Since this is true no matter what the behaviour of $\kappa$ along its left tail it must follow that $f_\kappa(\gamma)<\infty$ for some $\gamma>\beta$, and so 
\begin{align*}
\sum_{k\geq 0}e^{(2\theta+\beta) k}\frac{\nu(k)}{\log g(k)}<\infty\Rightarrow e^{(2\theta+\beta) k}\frac{\nu(k)}{\log g(k)}\rightarrow 0
\end{align*}
as $k\rightarrow\infty$ for some $\theta>0$. But $e^{\theta k}\nu(k)\rightarrow\infty$, so it follows that 
\begin{align*}
\frac{e^{(\theta+\beta) k}}{\log g(k)}\rightarrow 0.
\end{align*}
Thus there exists $K\geq0$ such that
for all $k>K$, $\log g(k)>e^{(\theta+\beta) k}$. Then for any $\lambda\geq 0$
\begin{align*}
E(e^{\lambda e^{(\theta+\beta )X_1}}) 
&\leq e^\lambda+\sum_{k\geq 0}e^{\lambda e^{(\theta+\beta) k}}\mu(k)\\
&=e^\lambda+\frac{1}{C_\nu}\sum_{k\geq 0}e^{\lambda e^{(\theta+\beta) k}}\frac{\nu(k)}{g(k)} \\ 
&\leq e^\lambda+\frac{1}{C_\nu}
\left(
\sum_{k=0}^Ke^{\lambda e^{(\theta+\beta) k}}\frac{\nu(k)}{g(k)}
+ \sum_{k>K} e^{\lambda e^{(\theta+\beta) k}}e^{-e^{(\theta+\beta) k}}\nu(k)
\right).
\end{align*}
For $0\leq\lambda<1$ the above series converges. As $E(e^{\lambda
e^{(\theta+\beta )X_\mu}})\leq 1$ for $\lambda<0$, $e^{(\beta+\theta)
X_1}$ has a moment generating function finite in a neighbourhood
of the origin. As finiteness of the moment generating function in
a neighbourhood of the origin is a sufficient condition for the
random walk associated with a distribution to satisfy the conditions
of Cram\'er's Theorem, we have that $(4)\Rightarrow (1)$.

In order to extend the result to an arbitrary distribution, we need the following notation. For any distribution $\nu$, define $\nu^*(k)=\nu([k+1))$ as the discretisation of $\nu$. Notice that if $X$ is a random variable with distribution $\nu$ and $X^*=\lfloor X\rfloor$ where $\lfloor\cdot\rfloor$ is the floor function, then $X^*$ has distribution $\nu^*$ and $X-1\leq X^*\leq X$, which implies that $\mathcal{D}_{f_\nu}=\mathcal{D}_{f_{\nu^*}}$.

Now assume (4) for some arbitrary distribution $\mu$, i.e. that $\H(\nu|\mu)=\infty$ for all $\nu$ with $f_\nu\in\effdom$. Consider $\mu^*$, and any $\kappa$ supported on a subset of the support of $\mu^*$ with $f_{\kappa}\in\effdom$. Then $\kappa(k)=g(k)\mu^*(k)$ for some non-negative function $g$ defined on $\mathbb{Z}$. Define $\nu(dx)=g(\lfloor x\rfloor)\mu(dx)$, then $\nu^*=\kappa$, and the effective domain of the moment generating function of $\nu$ is the same as that of $\nu^*=\kappa$, so $\H(\nu|\mu)=\infty$. Then 
\begin{align*}
\infty&=\int_{-\infty}^\infty\nu(dx)\log\frac{\nu(dx)}{\mu(dx)} 
=\int_{-\infty}^\infty g(\lfloor x\rfloor)\mu(dx)\log g(\lfloor x\rfloor) \\
&=\sum_{k=-\infty}^\infty g(k)\mu(k)\log g(k) 
=\sum_{k=-\infty}^\infty \kappa(k)\log \left(\frac{\kappa(k)}{\mu^*(k)}\right).
\end{align*} 
So $\H(\kappa|\mu^*)=\infty$. This is true for any $\kappa$ supported
on a subset of the support of $\mu^*$ with $f_{\kappa}\in\effdom$.
As any $\kappa$ that is not supported on a subset of the support
of $\mu^*$ also satisfies $\H(\kappa|\mu^*)=\infty$,
Lemma \ref{thm:5equiv} (4) holds for $\mu^*$ and so as we already established $(4)\Rightarrow (1)$
for discrete distributions, if $X_{\mu^*}$ is a random variable with distribution $\mu^*$ then the moment generating
function of $e^{\gamma X_{\mu^*}}$ is finite in a neighbourhood of
the origin for some $\gamma\not\in[\alpha,\beta]$. By arguments
similar to those showing domain equivalence of MGFs of discretised and
non-discretised distributions, it follows that the moment
generating function of $e^{\gamma X_{1}}$ is also finite in a
neighbourhood of the origin for the same $\gamma\not\in[\alpha,\beta]$.
As finiteness of the moment generating function in a neighbourhood
of the origin is a sufficient condition for the random walk associated
with a distribution to satisfy the conditions of Cram\'er's Theorem, we have that
$(4)\Rightarrow (1)$.

The proof is now almost complete. $(1)\Rightarrow (2)\Rightarrow
(3)\Rightarrow(4)\Rightarrow (1)$. That $(5)\Rightarrow (4)$ is
readily seen from Proposition
\ref{thm:I=H}. But then $(4)\Rightarrow (2)$ by the equivalence
of (1)-(4), and $(2)\Rightarrow (5)$ as $\effdom\subset\effdomsub$,
so that $(5)\Leftrightarrow (4)$ and we have equivalence of all 5
statements.

\end{proof}

\begin{proof}[{\bf Corollary \ref{cor:mineffdom}.}]
$\;$
\begin{enumerate}
\item Let $\beta_0$ be the supremum over all values of $\gamma$ such
that the random walk associated with $e^{\gamma X_1}$ satisfies
the conditions of Cram\'er's Theorem, and let $\alpha_0$ be the
infimum. Note that $\beta_0\geq 0$, $\alpha_0\leq 0$. Let
$\beta\geq\beta_0$ , $\alpha\leq\alpha_0$, and let $f_\nu$ satisfy
$\IM(f_\nu)<\infty$ and $[\alpha,\beta]\subset\mathcal{D}_{f_\nu}$.
Then for any $f\in\effdom$ mimicking $f_\nu$, $\IM(f)=\IM(f_\nu)$
as  Lemma \ref{thm:5equiv} (1) does not hold for $\alpha,\beta$, so that Lemma
\ref{thm:5equiv} (5) does not hold, and so we can apply statements
\ref{lemma:finitemimic} and \ref{lemma:valuesformimics} of Lemma
\ref{lemma:mimics}. If $\alpha>\alpha_0$ or $\beta<\beta_0$, Lemma
\ref{thm:5equiv} (1) holds, so that Lemma \ref{thm:5equiv} (5) holds
and so any $f\in\effdom$ mimicking $f_\nu$ satisfies
$\IM(f)=\infty\neq\IM(f_\nu)$. 
\item Let $f$ be such that $\IM(f)<\infty$, $f\in\effdom$. Then Lemma
\ref{thm:5equiv} (5) does not hold, so that Lemma \ref{thm:5equiv}
(1) does not hold, so that
$\overline{\effdomf}=[\alpha,\beta]\supset[\alpha_0,\beta_0]$. 
\item 
For $n$ sufficiently large, let $\mu_n$ be $\mu$ conditioned on $[-n,n]$. If $\alpha_0=\beta_0=0$,
let $f_n\in\mathcal{D}_{[-\frac{1}{n},\frac{1}{n}]}$ mimic $f_{\mu_n}$ and
let $f$ be the MGF finite only at 0, then $f_n\rightarrow f$ and so
$\IM(f)\leq\lim_{n\rightarrow\infty}\IM(f_n)=0$. If
$[\alpha_0,\beta_0]\neq \{0\}$, then letting $\alpha,\beta=0$ Lemma
\ref{thm:5equiv} (1) holds for some $\gamma\neq 0$ so that Lemma
\ref{thm:5equiv} (5) holds, and so $\IM(f)=\infty$.
\end{enumerate}
\end{proof}

\begin{proof}[{\bf Proposition \ref{thm:Iformimics}. $I_M(f)$ for $f$ mimics.}]
$\;$
\begin{enumerate}[(a)]
\item Let $[\alpha,\beta]=[0,0]$ in the statement of Lemma
\ref{thm:5equiv}. Then we have that for $f\in\MSpacesmall$ finite
only at 0
\begin{align*}
\IM(f)=\infty\Leftrightarrow\inf_{\nu:f_\nu=f}\H(\nu|\mu)=\infty
\end{align*}
using $(4)\Leftrightarrow (5)$. We have already shown that
$\IM(f)\in\{0,\infty\}$ by Lemma \ref{lemma:mimics}; we will show the same for the other term. See that for
any $\nu$ with $f_\nu=f$, $\kappa=a\nu+(1-a)\mu$ satisfies $f_\kappa=f$
for any $a\in(0,1]$, so that if $\H(\nu|\mu)<\infty$ for some such
$\nu$,
\begin{align*}
\inf_{\nu:f_\nu=f}\H(\nu|\mu)\leq\H(\kappa|\mu)\leq a\H(\nu|\mu)\rightarrow 0
\end{align*}
as $a\rightarrow 0$ by the convexity of $\H(\cdot|\mu)$, so $\inf_{\nu:f_\nu=f}\H(\nu|\mu)\in\{0,\infty\}$ also and equality holds. 
\item See first that $[\alpha,\beta]\neq[0,0]$ if $f\in\effdom$ is not a MGF but mimics one. Note that for any $g\in\effdom$ and $\mu_n(dx)=\mu(dx)/\mu([-n,n])$ for any $n$, $\gamma g+(1-\gamma)f_{\mu_n}\in\effdom$ for any $\gamma\in(0,1]$, so if $\IM(g)<\infty$ for some $g\in\effdom$
\begin{align*}
\inf_{\{g\in \effdom,g\text{ a MGF}\}}\IM(g)\leq \IM(\gamma g+(1-\gamma)f_{\mu_n})&\leq\gamma \IM(g)+(1-\gamma)\IM(f_{\mu_n})\rightarrow \IM(f_{\mu_n})
\end{align*}
as $\gamma\rightarrow 0$. By letting $n\rightarrow\infty$ we get
$\inf_{\{g\in \effdom,g\text{ a MGF}\}}\IM(g)=0$, and so in general
\begin{align*}
\inf_{\{g\in \effdom,g\text{ a MGF}\}}\IM(g)\in\{0,\infty\}.
\end{align*}
For the 0 value equality holds in the statement of
Proposition \ref{thm:Iformimics}(b) by statements \ref{lemma:finitemimic} and \ref{lemma:valuesformimics} of Lemma \ref{lemma:mimics}, and for the $\infty$ value equality holds by Lemma \ref{thm:5equiv} $(4)\Rightarrow (5)$, as by Proposition \ref{thm:I=H} (4) is equivalent to the expression $\inf_{\{g\in \effdom,g\text{ a MGF}\}}\IM(g)$ taking the value $\infty$.
\end{enumerate}
\end{proof}
\subsection{Section 3}
\begin{proof}[{\bf Theorem \ref{thm:MLDP}. LDP for large deviation estimates.}]
By Theorem \ref{thm:MLDPnonchar} we have that a LDP holds in $\MSpacesmall$.
Proposition \ref{thm:I=H} contains part (a), while Proposition
\ref{thm:Iformimics} is equivalent to parts part (b) and (c). The
subsequent statements about $\{\estLambda\}$, $\{\estJ\}$ and $\{\estI\}$ are
proved in Theorems \ref{thm:LLDP}, \ref{thm:JLDP} and \ref{thm:ILDP} respectively.
\end{proof}

\begin{proof}[{\bf Corollary \ref{cor:weaklaw}. Weak laws.}]
As $\{\estM\}$ satisfies a LDP in $(\MSpacesmall,\taw)$, we know that
$ P(\estM\in \cdot)$ is eventually concentrated on the closed set
$A_0=\{f:\IM(f)=0\}$ \cite{Lewis95}, i.e. $ P(\estM\in G)\rightarrow
1$ for every open $G$ containing $A_0$. First note that $f\in A_0$
implies that $f(\theta)=f_\mu(\theta)$ for all $\theta\in\effdomf$,
as $\H(\nu|\mu)$ has a unique zero at $\nu=\mu$. See that
$C_x^\gamma=\{f:f(\gamma)\leq x\}$ is a closed set for every $x$
and every $\gamma$ by a proof identical to that of Proposition \ref{prop:Mcomp}, so fix some $\gamma$ such that $f_\mu(\gamma)=
E(e^{\gamma X_1})<\infty$, and let $x>f_\mu(\theta)$. Then by the
weak law of large numbers
\begin{align*}
 P(\estM\in C_x^\gamma)= P(\estM(\gamma)\leq x)\rightarrow 1
\end{align*}
and so $ P(\estM\in \cdot)$ is eventually concentrated on
$C_x^\gamma$. Also note that since $\MSpacesmall$ is a Normal space,
disjoint closed sets have disjoint open neighbourhoods and so we
can show that $ P(\estM\in\cdot)$ is eventually concentrated
on $C_x^\gamma\cap A_0$, in the following way. Fix any neighbourhood
$G$ of $C_x^\gamma\cap A_0$, and note that $C_x^\gamma\backslash
G$ and $A_0\backslash G$ are disjoint closed sets. Then they have
disjoint open neighbourhoods $U$ and $V$ respectively, and so $G\cup
U$ and $G\cup V$ are open neighbourhoods of $C_x^\gamma$ and $A_0$
respectively, with intersection $G$. Thus
$P(\estM\in G\cup U)\rightarrow 1, P(\estM\in G\cup V)\rightarrow 1$, 
and so by the inequality 
\begin{align*}
1\geq  P(\estM\in (G\cup U)\cup(G\cup V))= P(\estM\in G\cup U)+ P(\estM\in G\cup V)- P(\estM\in G)
\end{align*}
we have that $P(\estM\in G)\rightarrow 1$.
Thus $P(\estM\in\cdot)$ is eventually concentrated on $C_x^\gamma\cap
A_0$, which consists of functions $f$ satisfying $f(\theta)=f_\mu(\theta)$
everywhere $f$ is finite, which includes $\gamma$. We can similarly
show that the measures are eventually concentrated on
$C_{x_1}^{\gamma_1}\cap C_{x_2}^{\gamma_2}\cap A_0$ if
$f_\mu(\gamma_1)<x_1$, $f_\mu(\gamma_2)<x_2$. Assume now that $f_\mu$
is finite in a neighbourhood of the origin. Fix some neighbourhood
$V_k(f_\mu)$ of $f_\mu$, and let $\gamma_+>0$ and $\gamma_-<0$ be
large enough so that $C_{x_+}^{\gamma_+}\cap C_{x_-}^{\gamma_-}\cap
A_0\subset V_k(f_\mu)$ for $x_+>f_\mu(\gamma_+)$, $x_->f_\mu(\gamma_-)$.
Then $V_k(f_\mu)$ is a neighbourhood of $C_{x_+}^{\gamma_+}\cap
C_{x_-}^{\gamma_-}\cap A_0$ and so $P(\estM\in V_k(f_\mu))\rightarrow 1$
as $n\rightarrow\infty$. This can be done for all $k$ and thus $
P(\estM\in \cdot)$ is eventually concentrated on the singleton set
$\{f_\mu\}$ and we have a weak law. If $f_\mu(\theta)=\infty$ for
all $\theta>0$ (resp. $\theta<0$), then in the above construction
we need only use $\gamma^-$ (resp. $\gamma^+$), and if $f_\mu$ is
finite only at 0 then $A_0$ is already a singleton set.

The results for $\{\estLambda\}$, $\{\estJ\}$ and $\{\estI\}$ follow from the
continuous mapping theorem, e.g. \cite{Billingsley99}[Theorem 2.7].

\end{proof}
\subsection*{Section 7}
\begin{proof}[{\bf Theorem \ref{thm:loynes}. LDP for Loynes' exponent estimates}]
First see that $\MSpacesmall$ is metrizable \cite{Beer93}, as is
$[0,\infty]$. Let $\delta_0$ be the Dirac measure at 0, let $G$ be
the Loynes' exponent mapping, and let $f\in\MSpacesmall$ neither
equal $f_{\delta_0}$ nor mimic it. It must follow that $f$ is equal
to 1 at at most one non-zero point in its domain. Therefore if
$G(f)\in[0,\infty)$, for $x>G(f)$ $d((x,1),\epi(f))>0$ and so for
any sequence $\{f_n\}$ converging to $f$ in $\taw$ $f_n(x)>1$ for
large enough $n$ and so $G(f_n)\leq x$. Similarly if $G(f)\in(0,\infty]$,
for any $0<x<G(f)$ $f(x)<1$ and as $\taw$ convergence implies
pointwise convergence on the interior of $\effdomf$, $f_n(x)<1$ for
large enough $n$ and so $G(f_n)\geq x$. Together this proves that
$G(f_n)\rightarrow G(f)$ and so $G$ is continuous at $f$. Therefore
$G^f$ is a singleton set containing only $G(f)$, the first condition
of Theorem \ref{Garcia} holds, and for the sequence $f_n=f$ for all
$n$ the second condition holds.

Now consider $f_{\delta_0}$. Assume $\IM(f_{\delta_0})<\infty$ as
otherwise we need not consider it. Note that this implies that $\mu$
has a point mass at 0. Fix any $y\in [0,\infty]$ and any MGF
$f_{\nu}\in\MSpacesmall$ finite everywhere satisfying $G(f_{\nu})=y$
and $\IM(f_{\nu})<\infty$. Letting $\nu$ equal some convex combination
of $\mu$ conditioned on $(-\infty,0)$ and $\mu$ conditioned on
$(0,M)$ for some $M$ will show that such an $f_{\nu}$ exists for
any $y\in[0,\infty]$. Let $\la=\alpha f_\nu+(1-\alpha)f_{\delta_0}$.
Then $\la\rightarrow f_{\delta_0}$ as $\alpha\rightarrow 0$ and
$G(\la)=y$ for all $\alpha\in(0,1)$, so that $G(\la)\rightarrow y$
as $\alpha\rightarrow 0$. This gives us that $G^{f_{\delta_0}}=[0,\infty]$,
and as $[0,\infty]$ is compact the first condition of Theorem
\ref{Garcia} holds trivially. Moreover as $\la=f_{\alpha
\nu+(1-\alpha)\delta_0}$,

\begin{align*}
\IM(\la)&=\H(\alpha \nu+(1-\alpha)\delta_0|\mu) \\
&=\int_{\mathbb{R}}(\alpha \nu(dx)+(1-\alpha)\delta_0(dx))\log\left(\frac{\alpha \nu(dx)+(1-\alpha)\delta_0(dx)}{\mu(dx)}\right) \\
&=\alpha\int_{\mathbb{R}\backslash\{0\}} \nu(dx)\log\left(\frac{\alpha \nu(dx)}{\mu(dx)}\right)+(\alpha\nu(0)+(1-\alpha)\delta_0(0))\log\left(\frac{\alpha \nu(0)+(1-\alpha)\delta_0(0)}{\mu(0)}\right) \\
&\rightarrow \delta_0(0)\log\left(\frac{\delta_0(0)}{\mu(0)}\right)\text{ as }\alpha\rightarrow 0 \\
&=\H(\delta_0|\mu) \\
&=\IM(f_{\delta_0}).
\end{align*}
Taking the limit has the desired result as $\H(\nu|\mu)<\infty$.
Therefore for every $y\in G^{f_{\delta_0}}$, the second condition
of Theorem \ref{Garcia} holds. If $g$ mimics $f_{\delta_0}$ and
$\IM(g)<\infty$, say with $\sup\mathcal{D}_g=b$ then the proof is
similar. For any $y\in[0,b]$ the same $f_{\nu}$ can be chosen, so
that $G^{g}$ contains $[0,b]$. As any sequence converging to $g$
must obey $f_n(\theta)\rightarrow\infty$ for all $\theta>b$ it
follows that $\limsup_{n\rightarrow\infty}G(f_n)\leq b$ and so
$G^{g}=[0,b]$ and agin the first condition of Theorem \ref{Garcia}
holds trivially. With $\la=\alpha f_{\nu}+(1-\alpha)g$ we have that
$\la$ mimics the MGF $f_{\alpha \nu+(1-\alpha)\delta_0}$. As
$\IM(g)<\infty$ and $\mathcal{D}_g=\mathcal{D}_{\la}$ we can use
statements \ref{lemma:finitemimic} and \ref{lemma:valuesformimics}
of Lemma \ref{lemma:mimics} to show that $\IM(\la)=\IM(f_{\alpha
\nu+(1-\alpha)\delta_0})=\H(\alpha \nu+(1-\alpha)\delta_0|\mu)$.
Identical calculations follow as above to show that $\IM(\la)\rightarrow
\IM(f)$.

Therefore we have that for all $f\in\MSpacesmall$ satisfying
$\IM(f)<\infty$, the conditions of Theorem \ref{Garcia} hold, and
so we have that $\{\estLo\}$ satisfies the LDP in $[0,\infty]$ with
rate function
\begin{align*}
\ILo(x)&=\inf_{f:x\in G^{f}}\IM(f).
\end{align*}
As seen by the characterisations of $G^f$ above for all $f\in\MSpacesmall$, for $x\in[0,\infty]$ $x\in G^f$ if and only $G(f)=x$, $f=f_{\delta_0}$, or $f$ mimics $f_{\delta_0}$ and is finite at $x$, i.e., if and only if $f\in C_x$. Therefore 
\begin{align*}
\ILo(x)=\inf_{f\in C_x}\IM(f),
\end{align*}
as required.
\end{proof}

\begin{proof}[{\bf Lemma \ref{lemma:positive}. Positive on $(\theta_\mu,\infty{]}$}]
As $C_x$ is closed, $\ILo(x)=\IM(f)$ for some $f$ satisfying $f(x)\leq 1$. As $f_\mu(x)>1$, $f$ does not mimic $f_\mu$ and so $\IM(f)>0$.
\end{proof}
\begin{proof}[{\bf Proposition \ref{prop:loynesuniquezero}. Conditions for unique zero}]
Fix $x\in[0,\theta_\mu)$. If $\ILo(x)=0$, then $\ILo(x)=\IM(f)$ for some
$f\in\mathcal{D}_{(-\infty,x]}^\subset$ mimicking $f_\mu$ with
$f(x)\leq 1$, and as $\IM(f)=0$, by the contrapositive of Lemma \ref{thm:5equiv}$(5)\Rightarrow (1)$ Cram\'er's
Theorem does not hold for $e^{yX_1}$ for any $y>x$. If $\ILo(x)>0$,
then $\IM(f)=\infty$ for $f$ mimicking $f_\mu$ with
$\mathcal{D}_f=\mathcal{D}_{f_\mu}\cap(-\infty,x]$, and so
$\IM(g)=\infty$ for all $g$ satisfying
$\mathcal{D}_g=\mathcal{D}_{f_\mu}\cap(-\infty,x]$ by Theorem
\ref{thm:MLDP}(c). So by Lemma \ref{thm:5equiv}$(5)\Rightarrow (1)$ Cram\'er's Theorem
holds for $e^{yX_1}$ for some
$y\not\in\mathcal{D}_{f_\mu}\cap(-\infty,x]$. As this cannot be
true for any $y\not\in\mathcal{D}_{f_\mu}$, it must hold for some
$y>x$, proving the first statement.  Therefore if $e^{yX_1}$ satisfies the conditions of
Cram\'er's Theorem for all $y\in(0,\theta_\mu)$ then $\ILo(x)>0$ for all $x\in[0,\theta_\mu)$ and so $\ILo$ has a
unique zero, and if $e^{xX_1}$ does not satisfy the conditions
of Cram\'er's Theorem for some $x\in(0,\theta_\mu)$, then $e^{yX_1}$
does not satisfy the conditions of Cram\'er's Theorem for any $y>x$,
so $\ILo(x)=0$ and the second statement also holds.
\end{proof}

\begin{proof}[{\bf Theorem \ref{thm:propertiesofILo}. Properties of $\ILo$}]
$\;$
\begin{enumerate}[(a)]
\item Fix any finite $x>\theta_\mu$, and any function $g$ with
$g(x)\leq 1$. Then let $f_{\mu_n}$ be the moment generating function
of $\mu$ conditioned on $[-n,n]$. As discussed before
$\IM(f_{\mu_n})\rightarrow \IM(f_\mu)$, and as $f_\mu(x)>1$,
$f_{\mu_n}(x)>1$ for large enough $n$. For such an $n$,
\begin{align*}
ag(x)+(1-a)f_{\mu_n}(x)=1
\end{align*}
for some $a\in(0,1]$, and so 
\begin{align*}
\ILo(x)&\leq \IM(ag+(1-a)f_{\mu_n})\leq a\IM(g)+(1-a)\IM(f_{\mu_n})\leq \IM(g)+\IM(f_{\mu_n}) \\
\Rightarrow \ILo(x)&\leq \lim_{n\rightarrow\infty}(\IM(g)+\IM(f_{\mu_n}))=\IM(g).
\end{align*}
We used $f_{\mu_n}$ because we didn't know that $f_\mu$ was finite
at $x$. This is true for all $g$ with $g(x)\leq 1$, and so it follows
from the definition of $\ILo$ that
\begin{align*}
\ILo(x)=\inf_{f:f(x)\leq 1}\IM(f).
\end{align*}
Then for all $\theta_\mu<x<y$ 
\begin{align*}
\ILo(\theta_\mu)\leq\ILo(x)=\inf_{f:f(x)\leq 1}\IM(f)\leq\inf_{f:f(y)\leq 1}\IM(f)=\ILo(y)
\end{align*}
as $f(y)\leq 1\Rightarrow f(x)\leq 1$. Moreover
\begin{align*}
\ILo(x)=\inf_{f:f(x)\leq 1}\IM(f)\leq \inf_{f\in C_\infty}\IM(f)=\ILo(\infty),
\end{align*}
so $\ILo$ is increasing on $[\theta_\mu,\infty]$. For $x<\theta_\mu$ satisfying $\ILo(x)>0$, 
\begin{align}
\ILo(x)=\inf_{f:f(x)=1}\IM(x)\label{eq:niceinf}
\end{align}
as $\IM(f)=\infty$ for
all functions $f$ infinite on $y>x$ 
by Proposition \ref{prop:loynesuniquezero} and Lemma
\ref{thm:5equiv}. For any $g$ with $g(x)\geq 1$
\begin{align*}
ag(x)+(1-a)f_\mu(x)=1
\end{align*}
for some $a\in(0,1]$ and so 
\begin{align*}
\ILo(x)\leq \IM(ag+(1-a)f_\mu)\leq a\IM(g)+(1-a)\IM(f_\mu)\leq\IM(g).
\end{align*}
This is true for all $g$ with $g(x)\geq 1$ and so using (\ref{eq:niceinf})
\begin{align*}
\ILo(x)=\inf_{f:f(x)\geq 1}\IM(f).
\end{align*}
Similar to before we can now show that $\ILo(\theta_\mu)\leq\IM(x)\leq\IM(y)$ for any $y<x$. This is trivially true for $y<x<\theta_\mu$ with $\ILo(x)=0$. Moreover 
\begin{align*}
\ILo(x)=\inf_{f:f(x)\geq 1}\IM(f)\leq\inf_{f\in C_0}\IM(f)=\ILo(0),
\end{align*}
and so $\ILo$ is decreasing on $[0,\theta_\mu]$. \\
\item If $\mu_+, \mu_-$ are $\mu$ conditioned on $[0,\infty),(-\infty,0]$ respectively then $\ILo(0)\leq\IM(f_{\mu_+})<\infty$ and $\ILo(\infty)\leq\IM(f_{\mu_-})<\infty$, so together with part (a) this proves that $\ILo$ is finite everywhere and bounded. 
\item Let $\infty>x>\theta_\mu$. $\ILo(x)=\IM(f)>0$ for some $f\in C_x$. If $f(x)<1$ then for $f_{\mu_n}$ with $\mu_n$ the distribution $\mu$ conditioned on $(-\infty,n]$ we can choose $n$ large enough so that $f_{\mu_n}(x)>1$ and $\IM(f_{\mu_n})<\IM(f)$, as $\IM(f_{\mu_n})\rightarrow\IM(f_\mu)=0$. Moreover
\begin{align*}
af_{\mu_n}(x)+(1-a)f(x)=1
\end{align*}
for some $a\in(0,1)$ and so 
\begin{align*}
\ILo(x)\leq a\IM(f_{\mu_n})+(1-a)\IM(f)<\IM(f).
\end{align*}
So we cannot have $f(x)<1$, and therefore we must have $f(x)=1$. 
\item First let $\infty>x\geq\theta_\mu$. By monotonicity and lower semi-continuity, we need only show that 
\begin{align}
\lim_{\epsilon\downarrow 0}\ILo(x+\epsilon)\leq\ILo(x)\label{contlimit}
\end{align}
in order to show continuity at $x$. $\ILo(x)=\IM(f)$ for some $f\in C_x$. We may assume that $f$ is a moment generating function, as if $x>\theta_\mu$ then $f(x)=1$ by (c), so if it is mimicking a moment generating function $\fmimic$ then it also holds that $\fmimic\in C_x$. If $x=\theta_\mu$ then we can have $f=f_\mu$. So $f=f_\nu$ for some distribution $\nu$. Assume for now that $f_\nu(y)$ is finite for some $y>x$. Then $f_\nu(x+\epsilon)>1$ for all $\epsilon>0$, and for sufficiently small $\epsilon$ 
\begin{align*}
a_\epsilon f_\nu(x+\epsilon)+(1-a_\epsilon)\mum(x+\epsilon)=1
\end{align*}
for 
\begin{align*}
a_\epsilon=\frac{\mum(x+\epsilon)-1}{\mum(x+\epsilon)-f_\nu(x+\epsilon)}\rightarrow 1
\end{align*}
as $\epsilon\rightarrow 0$. Therefore 
\begin{align*}
\ILo(x+\epsilon)\leq \IM(a_\epsilon f_\nu+(1-a_\epsilon)\mum)\leq a_\epsilon\IM(f_\nu)+(1-a_\epsilon)\IM(\mum)\rightarrow \IM(f_\nu)=\IM(x)
\end{align*}
as $\epsilon\downarrow 0$. Now assume $f_\nu(y)=\infty$ for all $y>x$. Then the support of $\nu$ is not bounded above. Let $f_{\nu_n}$ be the moment generating function of $\nu$ conditioned on $(-\infty,n]$. Then $f_{\nu_n}$ is finite on $[0,\infty)$, $f_{\nu_n}(x)<f_\nu(x)$, and $f_{\nu_n}$ converges point-wise to $f_\nu$, or to $\infty$ where $f_\nu$ is infinite. Therefore if $\epsilon_n$ is such that $f_{\nu_n}(x+\epsilon_n)=1$ ($\epsilon_n$ may not exist for small $n$ if $\nu([0,n])=0$), then $\epsilon_n>0$ and $\epsilon_n\rightarrow 0$ as $n\rightarrow\infty$. Moreover
\begin{align*}
\ILo(x+\epsilon_n)&\leq \IM(f_{\nu_n}) \\
\Rightarrow \lim_{n\rightarrow\infty}\ILo(x+\epsilon_n)&\leq \lim_{n\rightarrow\infty}\IM(f_{\nu_n})=\IM(f_\nu).
\end{align*}
As $\ILo$ is increasing on $[\theta_\mu,\infty)$, to prove (\ref{contlimit}) it is sufficient to prove that $\lim_{n\rightarrow\infty}\ILo(x+\epsilon_n)\leq\ILo(x)$ for some positive sequence $\epsilon_n$ converging to 0. Now assume $0<x\leq\theta_\mu$. This time is is sufficient to show that 
\begin{align}
\lim_{\epsilon\downarrow 0}\ILo(x-\epsilon)\leq\ILo(x)\label{contlimit2}
\end{align}
in order to prove continuity at $x$. If $\IM(x)=0$ and moreover $\IM(y)=0$ for some $y<x$ then proof is trivial. So we need only prove continuity for $x$ satisfying $\IM(x)>0$. Then $\ILo(x)=\IM(f)>0$, so that $f(x)=1$ as shown in part (a). Then for all $\epsilon>0$, 
\begin{align*}
a_\epsilon f(x-\epsilon)+(1-a_\epsilon)\mup(x-\epsilon)=1
\end{align*}
for 
\begin{align*}
a_\epsilon=\frac{\mup(x-\epsilon)-1}{\mup(x-\epsilon)-f(x-\epsilon)}\rightarrow 1
\end{align*}
as $\epsilon\rightarrow 0$. Moreover, 
\begin{align*}
\ILo(x-\epsilon)&\leq a_\epsilon\IM(f)+(1-a_\epsilon)\IM(\mup) \\
\Rightarrow \lim_{\epsilon\downarrow 0}\ILo(x-\epsilon)&\leq \IM(f)=\ILo(x), 
\end{align*}
as required. Continuity at $\infty$ and at 0 is immediate by monotonicity and lower-semicontinuity, completing the proof. 
\end{enumerate}
\end{proof}
{\bf Acknowledgments:} The authors thank Gerald Beer (California
State University) for pointing them in the right direction in proving
the compactness in Proposition \ref{prop:Mcomp}. This work was
conducted while Brendan Williamson was visiting the National
University of Ireland Maynooth from Dublin City University and then
Duke University, for which funding is gratefully acknowledged.

\end{document}